\documentclass{amsart}

\usepackage{paralist}
\usepackage{color}
\usepackage{mathrsfs}
\usepackage{amssymb}
\usepackage{amsmath}
\usepackage{amsfonts}
\usepackage{stmaryrd}
\usepackage[all]{xy}

\makeatletter
\@namedef{subjclassname@2020}{\textup{2020} Mathematics Subject Classification}
\makeatother

\newtheorem{theorem}{Theorem}[section]
\newtheorem{lemma}[theorem]{Lemma}
\newtheorem{proposition}[theorem]{Proposition}
\newtheorem{corollary}[theorem]{Corollary}
\newtheorem{claim}[theorem]{Claim}

\newtheorem{question}[theorem]{Question}

\newtheorem{conjecture}[theorem]{Conjecture}

\theoremstyle{definition}
\newtheorem{definition}[theorem]{Definition}
\newtheorem{remark}[theorem]{Remark}

\usepackage[pdftex]{hyperref}
\hypersetup{
    colorlinks=true, 
    linktoc=all,     
    linkcolor=blue,  
    allcolors=blue,
}

\usepackage[framemethod=tikz]{mdframed}

\newcommand{\dom}{\mathrm{dom}}
\newcommand{\bb}{\mathbb}

\newcommand{\mc}{\mathcal}
\newcommand{\power}{\ensuremath{\mathscr{P}}}
\newcommand{\sub}{\subseteq}
\newcommand{\ul}{\underline}
\newcommand{\ra}{\rightarrow}

\newcommand{\Add}{\mathrm{Add}}

\newcommand{\R}{\bb{R}}
\newcommand{\Q}{\bb{Q}}
\renewcommand{\P}{\bb{P}}
\newcommand{\B}{\bb{B}}

\newcommand{\Z}{\bb{Z}}

\newcommand{\ZFC}{\sf ZFC}

\newcommand{\MA}{\sf MA}

\newcommand{\CondAb}{\mathsf{CondAb}}
\newcommand{\Ab}{\mathsf{Ab}}
\newcommand{\Cond}{\mathsf{Cond}}

\newcommand{\Ker}{\mathrm{Ker}}
\newcommand{\Ext}{\mathrm{Ext}}
\newcommand{\Hom}{\mathrm{Hom}}
\newcommand{\iExt}{\underline{\mathrm{Ext}}}
\newcommand{\iHom}{\underline{\mathrm{Hom}}}

\title{Whitehead's problem and condensed mathematics}

\author{Jeffrey Bergfalk}
\address[Bergfalk]{Departament de Matem\`{a}tiques i Inform\`{a}tica \\
Universitat de Barcelona \\
Gran Via de les Corts Catalanes, 585 \\ 08007 Barcelona, Catalonia}
\email{bergfalk@ub.edu}

\author{Chris Lambie-Hanson}
\address[Lambie-Hanson]{
Institute of Mathematics, 
Czech Academy of Sciences, 
{\v Z}itn{\'a} 25, Prague 1, 
115 67, Czech Republic
}
\email{lambiehanson@math.cas.cz}
\urladdr{https://clambiehanson.github.io}

\author{Jan \v{S}aroch}
\address[\v{S}aroch]{
Charles University, 
Faculty of Mathematics and Physics,
Department of Algebra,
Sokolovsk{\'a} 83, Prague 8, 
186 75, Czech Republic
}
\email{saroch@karlin.mff.cuni.cz}


\thanks{The first author was supported by Mar\'{i}a Zambrano and Marie Sk\l odowska Curie (project 101110452: CatT) Fellowships at the University of Barcelona; the second author was supported by GA\v{C}R project 23-04683S 
and the Academy of Sciences of the Czech Republic (RVO 67985840); the third author was supported by GA\v{C}R project\ 23-05148S}

\begin{document}

\begin{abstract}
One of the better-known independence results in general mathematics is Shelah's solution to Whitehead's problem of whether $\mathrm{Ext}^1(A,\mathbb{Z})=0$ implies that an abelian group $A$ is free.
The point of departure for the present work is Clausen and Scholze's proof that, in contrast, one natural interpretation of Whitehead's problem within their recently-developed framework of condensed mathematics has an affirmative answer in $\ZFC$.
We record two alternative proofs of this result, as well as several original variations on it, both for their intrinsic interest and as a springboard for a broader study of the relations between condensed mathematics and set theoretic forcing.
We show more particularly how the condensation $\underline{X}$ of any locally compact Hausdorff space $X$ may be viewed as an organized presentation of the forcing names for the points of canonical interpretations of $X$ in all possible set-forcing extensions of the universe, and we argue our main result by way of this fact.
We also show that when interpreted within the category of light condensed abelian groups, Whitehead's problem is again independent of the $\ZFC$ axioms.
In fact we show that it is consistent that Whitehead's problem has a negative solution within the category of $\kappa$-condensed abelian groups for every uncountable cardinal $\kappa$, but that this scenario, in turn, is inconsistent with the existence of a strongly compact cardinal.
\end{abstract}

\maketitle

\section{Introduction}

Recall that an abelian group $A$ is \emph{Whitehead} if $\Ext^1(A, \Z) = 0$. Around 1950, J.H.C.\ Whitehead 
asked whether every such group is free. The converse is easily seen to be true (every free abelian 
group $A$ satisfies $\Ext^1(A,G) = 0$ for \emph{every} abelian group $G$, and hence for $G = \Z$), and 
Stein showed soon after the question was posed that every \emph{countable} Whitehead group is free 
\cite{stein}.
Progress on the general question was slow, though, and the latter soon assumed the status of a major open problem in homological algebra.
Shelah's solution to the question 20 years later was completely unexpected, for Whitehead's problem thereby became one of the first originating unequivocally outside of logic or set theory to be proven independent of the $\ZFC$ axioms.

\begin{theorem}[Shelah, \cite{Shelah_infinite_74}] \label{shelah_thm}
  Whitehead's problem is independent of $\ZFC$. In particular:
  \begin{enumerate}
    \item If $V = L$, then every Whitehead group is free.
    \item If $\MA_{\omega_1}$ holds, then there is a nonfree Whitehead group of size $\aleph_1$.
  \end{enumerate}
\end{theorem}
Shelah's solution of Whitehead's problem initiated a program 
of research applying set 
theoretic methods to module theory that continues to bear fruit 
today. This program has uncovered deep connections between Whitehead's problem and its variations and purely combinatorial set theoretic phenomena, and has touched on additional topics including the study of 
almost free modules and of the approximation 
theory of modules. For more on such developments in the 
intervening years, see, e.g., \cite{Eklof_Mekler, gobel_trlifaj, Magidor_Shelah}.

Recently, Clausen and Scholze have developed a framework, known as 
\emph{condensed mathematics},
for applying algebraic tools to algebraic structures 
endowed with nontrivial topologies (cf.\ \cite{CS1}, \cite{CS2}, 
\cite{CS3}). Here, we will be particularly interested 
in the category $\CondAb$ of \emph{condensed abelian 
groups}, where the objects are, roughly speaking, certain 
contravariant functors from the category of Stone spaces to 
the category of abelian groups. Every T1 topological abelian group $A$ has a corresponding 
condensed abelian group $\ul{A}$ (if an abelian group is given 
without specifying its topology, then we interpret it as a 
topological abelian group with the discrete topology). 
In addition, $\CondAb$ has an internal $\Hom$ functor 
$\iHom : \CondAb^{\mathrm{op}} \times \CondAb \ra \CondAb$, 
and corresponding internal $\Ext$ functors $\iExt^n$ for 
$n > 0$. It turns out that, when interpreted appropriately 
in $\CondAb$, Whitehead's problem is no longer independent of 
$\ZFC$. In particular, Clausen and Scholze proved the following 
theorem (cf.\ \cite{masterclass}).

\begin{theorem}[Clausen--Scholze, 2020] \label{cs_whitehead_thm}
  It follows from the $\ZFC$ axioms that, for every abelian group $A$, if 
  \[
    \iExt^1(\ul{A}, \ul{\bb{Z}}) = 0,
  \]
  then $A$ is free.
\end{theorem}

Clausen and Scholze's proof of Theorem \ref{cs_whitehead_thm} is 
quite slick and relies on a structural analysis of the
abelian subcategory of $\CondAb$ spanned by what are known as 
\emph{solid} abelian groups. In this paper, we present a more 
elementary, combinatorial proof of Theorem \ref{cs_whitehead_thm} 
that is also more \emph{constructive} in the sense that, given 
a nonfree abelian group $A$, we specify a highly natural associated 
Stone space $S$ and construct a nonzero element of 
$\iExt^1(\ul{A}, \ul{\Z})(S)$.
Our motivation for doing this is twofold. First, we seek to 
excavate some of the combinatorial and set theoretic ideas 
underlying Theorem \ref{cs_whitehead_thm} and thereby to improve 
our understanding of the mathematics surrounding Whitehead's problem, in 
both classical and condensed contexts. To this end, we prove 
the following theorem, item (2) of which strengthens Theorem 
\ref{cs_whitehead_thm}.

\begin{theorem} \label{main_thm}
    Suppose that $A$ is a nonfree abelian group and $\kappa$ is 
    the least cardinality of a nonfree subgroup of $A$. Let 
    $\B$ be the Boolean completion of the forcing to add 
    $\kappa$-many Cohen reals, and let $S$ be the Stone 
    space of $\B$. Then:
    \begin{enumerate}
        \item in $V^{\B}$, $A$ is not Whitehead; and
        \item $\iExt^1(\ul{A},\ul{\Z})(S) \neq 0$.
    \end{enumerate}
\end{theorem}

Second, using Theorem \ref{main_thm} as a starting point, we begin 
to investigate a more general connection between condensed 
mathematics and set theoretic forcing. We isolate a precise 
sense in which, given a locally compact Hausdorff space $X$, 
the condensed set $\ul{X}$ is simply an organized presentation 
of all forcing names for points in the canonical 
interpretations of $X$ in all possible set forcing extensions.
We then exhibit a correspondence between the satisfaction of 
certain simple formulas concerning a topological space 
$X$ in two settings: first, in a Boolean-valued model 
$V^{\B}$, and second in the context of the space 
$C(S,X)$ of continuous functions from the Stone space $S$ of $\B$ to $X$. This correspondence, stated in 
Theorem \ref{equivalence_thm}, allows us, e.g., to immediately 
deduce clause (2) of Theorem~\ref{main_thm} from clause (1) 
of the same theorem.

It is natural, in light of the preceding paragraph, to wonder whether there are provable implications between instances of clauses (1) and (2) of 
Theorem \ref{main_thm} for more general complete Boolean algebras $\B$.
We record several partial answers to this question.
In 
particular, in Theorem \ref{forcing_thm} we show that if 
$\B$ is a complete Boolean algebra, $S$ is its Stone space, and 
$V^{\B}$ satisfies a certain technical strengthening of the 
failure of $A$ to be Whitehead, then $\iExt^1(\ul{A}, \ul{\Z})(S) 
\neq 0$. In the ensuing discussion, we show that the converse fails 
in general.
The question of whether clause (1) more generally implies clause (2) (i.e., of whether the aforementioned technical strengthening can be dropped) remains open and is recorded as Conjecture \ref{conjecture}.

We also investigate some variations of Theorem \ref{cs_whitehead_thm}. 
First, we show that, if we restrict our attention to the subcategory 
$\CondAb_\kappa$ of $\CondAb$ for certain uncountable cardinals 
$\kappa$, then the classical independence of Whitehead's 
problem remains, while it again vanishes when $\kappa$ is a sufficiently 
large cardinal. More precisely, we prove the following theorem. 
(In what follows, given an uncountable cardinal $\kappa$, $\CondAb_\kappa$ 
denotes the category of $\kappa$-condensed sets, with the particular 
case of $\CondAb_{\omega_1}$ denoting the category of \emph{light} 
condensed abelian groups, the ambient framework for the lecture 
series ``Analytic Stacks'' delivered by Clausen and Scholze in 
2023--24 \cite{analytic_stacks}).

\begin{theorem} \label{local_theorem}
\leavevmode
    \begin{enumerate}
        \item If $\MA_{\omega_1}$ holds, then there is a nonfree abelian 
        group of size $\aleph_1$ such that 
        $\iExt^1_{\CondAb_{\omega_1}}(\ul{A}, \ul{\Z}) = 0$.
        \item It is consistent with the axioms of $\ZFC$ that 
        for every uncountable cardinal~$\kappa$, there is a nonfree 
        abelian group $A$ such that $\iExt^1_{\CondAb_\kappa}(\ul{A},\ul{\Z}) 
        = 0$.
        \item If $\kappa$ is a strongly compact cardinal, $A$ is an abelian 
        group, and 
        \[
        \iExt^1_{\CondAb_\kappa}(\ul{A},\ul{\Z}) = 0,
        \]
        then $A$ is free.
    \end{enumerate}
\end{theorem}

Second, we extend Theorem \ref{cs_whitehead_thm} from the 
category of (discrete) abelian groups to the category of locally compact 
abelian groups; here it is useful to recall that the projective objects of the former category are exactly the free ones. In particular, we prove the following theorem:

\begin{theorem} \label{thm:LCA_intro}
  Let $A$ be an object of the category $\mathsf{LCA}$ of locally compact abelian groups. 
  Then
	\[
	  \underline{\mathrm{Ext}}^1_{\CondAb}(\ul{A},\ul{\mathbb{Z}})=0
	\]
  if and only if $A$ is projective in $\mathsf{LCA}$.
\end{theorem}
We record as well several observations by Peter Scholze showing that Theorem \ref{cs_whitehead_thm} does not extend to the category of solid abelian groups.

The structure of the paper is as follows. In Section 
\ref{condensed_sec}, we introduce the basic definitions and 
facts regarding condensed mathematics that we will need. 
In Section \ref{prelim_sec}, we recall some classical 
algebraic facts pertinent to Whitehead's problem and introduce 
their analogues in the condensed setting. In 
Section \ref{countable_sec}, we prove a~slight strengthening of Stein's aforementioned classical result \cite{stein} that all 
countable Whitehead groups are free. In Section 
\ref{unctble_sec}, we prove a precursor to Theorem \ref{main_thm} concerning the 
space $2^\kappa$, 
and, in Section \ref{cohen_sec}, we adapt that proof to establish 
clause (1) of Theorem \ref{main_thm}. From here we present two 
paths to clause (2) of Theorem \ref{main_thm}. The first, recorded 
in Section \ref{stone_sec}, is by far the shorter and more direct 
of the two. The second, comprising Sections \ref{top_spaces_sec} 
and \ref{compact_sec}, passes through a more general investigation 
into the interplay between condensed mathematics and forcing.
In Section \ref{top_spaces_sec}, a robust correspondence is 
established between 
\begin{itemize}
    \item $\B$-names for points in a certain extension of a 
    topological space $X$; and
    \item continuous functions from the Stone space $S$ of 
    $\B$ to $X$.
\end{itemize}
Section \ref{compact_sec} then establishes the aforementioned 
equivalence between the settings of $C(S,X)$ and $V^{\bb{B}}$
with respect to interpretations of certain formulas. We expect the 
results of this section to be applicable in a wide variety of 
contexts, and derive clause (2) of Theorem \ref{main_thm} as an immediate corollary.
In Section \ref{ma_sec}, we prove Theorem \ref{local_theorem}(1). 
In Section \ref{alternate_section}, we record an algebraic lemma 
converting the groups $\iExt^1_{\CondAb}(\ul{A},\ul{\Z})(S)$ into a more tractable, and suggestive, form. This lemma yields an alternative 
proof of clause (2) of Theorem \ref{main_thm} (and hence of Theorem~\ref{cs_whitehead_thm}), 
as well as clauses (2) and (3) of Theorem \ref{local_theorem}.
In Section \ref{lca_sec} we prove Theorem \ref{thm:LCA_intro} and record impediments to 
the extension of Theorem \ref{cs_whitehead_thm} to the class of solid abelian groups. We close out our work in Section \ref{conclusion} with 
a few concluding remarks and one main open question.

We have pursued wherever possible an account of our results accessible 
to readers with no more than a basic knowledge of either set theory 
or homological algebra. With this in mind, let us record here the 
places in the paper where more background knowledge 
may be required. First, some of the arguments in Section 
\ref{prelim_sec} presume some knowledge of derived functors and 
derived categories. That said, at the end of Section 
\ref{prelim_sec} we isolate an elementary combinatorial statement 
equivalent to the assertion that $\ul{A}$ is not Whitehead for 
a given abelian group $A$. It is this combinatorial statement that 
we work with for most of the remainder of the paper, so the reader 
may safely take the calculations of Section \ref{prelim_sec} 
as a black box. Sections \ref{cohen_sec}--\ref{compact_sec} 
require some basic knowledge of forcing. Both \cite{jech} 
and \cite{kunen} contain more than is needed here; we also direct 
the reader to \cite{moore_forcing}, which provides a concise 
introduction to forcing aimed at non-specialists. 
The final sections are slightly more demanding, with Section 
\ref{ma_sec} requiring a somewhat more sophisticated 
understanding of forcing and Sections \ref{alternate_section} and 
\ref{lca_sec} requiring some knowledge of homological algebra and 
condensed mathematics.


\subsection{Notation and conventions}

Our standard references for undefined notions and notations are 
\cite{jech} for set theory and \cite{weibel} for homological 
algebra.

Given a group $G$, we will sometimes slightly abuse notation and write $G$ to refer to its underlying 
set. We let $|G|$ denote the cardinality of the underlying set of $G$.
If $f$ is a function and $X \subseteq \dom(f)$, then $f[X]$ denotes 
the image of $X$ under $f$, i.e., $f[X] = \{f(x) \mid x \in X\}$.

Given a Boolean algebra $\B$, we let $\B^+$ denote the set of all nonzero elements 
of $\B$ and $S(\B)$ denote the Stone space 
of $\B$. Concretely, $S(\B)$ is the space of ultrafilters on $\B$, 
with basic clopen sets given by 
\[
  N_b := \{\mc{U} \in S(\B) \mid b \in \mc{U}\}
\]
for $b \in \B^+$.

\section{Condensed abelian groups} \label{condensed_sec}

Here we briefly introduce the category $\CondAb$ of condensed 
abelian groups, as well as the category $\Cond$ of condensed sets. 
For more details and proofs of many of the statements contained in 
this section, we refer the reader to \cite{CS1}. We let
$\mathsf{CHaus}$, $\mathsf{Prof}$, and $\mathsf{ED}$ denote the 
categories of compact Hausdorff spaces, profinite sets (i.e., 
totally disconnected compact Hausdorff spaces, also known as Stone spaces), and extremally 
disconnected compact Hausdorff spaces, respectively, noting that 
$\mathsf{ED} \subseteq \mathsf{Prof} \subseteq \mathsf{CHaus}$. 
Recall that extremally disconnected compact Hausdorff spaces are precisely 
those spaces homeomorphic to Stone spaces of complete Boolean algebras, and 
that profinite sets are precisely those spaces $S$ that can 
be represented as inverse limits $S = \varprojlim_{i \in I} S_i$, 
where $I$ is a directed partial order and each $S_i$ is a finite 
(discrete) space. Given an uncountable cardinal $\kappa$, we say 
that a profinite set $S$ is \emph{$\kappa$-small} if it can be 
represented as an inverse limit as above with $|I| < \kappa$. 
Equivalently, a profinite set $S$ is $\kappa$-small if the set 
of all of its clopen subsets (or, equivalently, its topological weight) has cardinality less than $\kappa$.\footnote{
There is some discrepancy in the terminology around this notion in 
various presentations of condensed mathematics. In \cite{CS1}, 
$\kappa$-small profinite sets are defined to be those profinite sets 
of \emph{cardinality} less than $\kappa$, while in the later \cite{CS3}, 
$\kappa$-small profinite sets are defined to be those profinite sets 
with fewer than $\kappa$-many clopen subsets. This distinction is 
immaterial for strong limit cardinals $\kappa$, but in general we find 
the latter convention more natural and thus adopt it here.} 
We let $\mathsf{Prof}_\kappa$ denote the full subcategory of 
$\mathsf{Prof}$ consisting of all $\kappa$-small profinite sets, 
and define $\mathsf{ED}_\kappa$ in the same way. We write $S\times_T S'$ for the fiber product of a pair of morphisms $f:S\to T$ and $g:S'\to T$ in any of these categories; concretely, $S\times_T S'=\{(x,y)\in S\times S'\mid f(x)=g(y)\}$.
The category of 
sets is denoted $\mathsf{Set}$, and the category of abelian groups 
is denoted $\Ab$. In any category with a terminal object, we let 
$\ast$ denote that terminal object (so, e.g., $\ast$ is the 
one-point space in $\mathsf{CHaus}$, the one-point set in 
$\mathsf{Set}$, and the one-element group in $\mathsf{Ab}$).

\begin{definition}
\label{def:kappa_cond_set}
    Let $\kappa$ be an uncountable cardinal. A \emph{$\kappa$-condensed 
    set} (resp.\ \emph{$\kappa$-condensed abelian group}) is a contravariant functor $T: \mathsf{Prof}_\kappa \ra \mathsf{Set}$ 
    (resp.\ $T: \mathsf{Prof}_\kappa \ra \mathsf{Ab}$) such that
    the following conditions hold.
    \begin{enumerate}
        \item $T(\emptyset) = \ast$.
        \item For all $S_0, S_1 \in \mathsf{Prof}_\kappa$, the natural 
        map
        \[
          T(S_0 \sqcup S_1) \ra T(S_0) \times T(S_1)
        \]
        induced by the inclusion maps of $S_0$ and $S_1$ into $S_0 \sqcup S_1$ 
        is a bijection.
        \item Suppose that $S, S' \in \mathsf{Prof}_\kappa$ and 
        $f:S' \ra S$ is a surjective continuous map. Let $p_0$ and 
        $p_1$ denote the two projections from $S' \times_S S'$ to 
        $S'$, and $p_0^*$ and $p_1^*$ their images under $T$. 
        Then the natural map
        \[
          T(S) \ra \{x \in T(S') \mid p_0^*(x) = p_1^*(x)\}
        \]
        induced by $f$ is a bijection.\footnote{In slightly more detail, 
        $f$ induces a map from $T(S)$ to $T(S')$ whose range is readily seen 
        to be contained in the set $\{x \in T(S') \mid p_0^*(x) = p_1^*(x)\}$. 
        This condition then asserts that this map is a bijection between 
        $T(S)$ and $\{x \in T(S') \mid p_0^*(x) = p_1^*(x)\}$.}
    \end{enumerate}
    More succinctly, a $\kappa$-condensed set (resp. 
    $\kappa$-condensed abelian group) is a sheaf of 
    sets (resp. sheaf of abelian groups) on the site 
    $\ast_{\kappa\text{-pro\'{e}t}}$.
    The category of $\kappa$-condensed sets is denoted $\Cond_\kappa$, 
    and the category of $\kappa$-condensed abelian groups is denoted 
    $\CondAb_\kappa$.
\end{definition}


Three sorts of choices of $\kappa$ in Definition \ref{def:kappa_cond_set} 
have predominated so far:
\begin{itemize}
    \item strong limit cardinals $\kappa$, valued for their closure properties in \cite{CS1, CS2, CS3}, with the associated categories $\Cond_\kappa$ assembling to form $\Cond$ in the fashion described just below;
    \item inaccessible cardinals $\kappa$, valued for their even stronger closure properties in the closely related and contemporaneous \emph{pyknotic} framework of Barwick and Haine \cite{barwick_haine};
    \item $\kappa=\omega_1$, valued for its relative minimality in \cite{analytic_stacks}, wherein the associated condensed sets are termed \emph{light}. We will return to this setting in Section \ref{ma_sec}. Note that $\mathsf{Prof}_{\omega_1}$ is precisely the category of totally disconnected \emph{metrizable} compact Hausdorff spaces.
\end{itemize}

Given strong limit cardinals 
$\kappa < \kappa'$, the forgetful functors $\Cond_{\kappa'} 
\ra \Cond_\kappa$ and $\CondAb_{\kappa'} \ra \CondAb_\kappa$ 
have natural left adjoints; by way of these adjoints, the categories $\Cond_\kappa$ form a direct system, where $\kappa$ ranges over the strong limit cardinals 
$\kappa$; one then defines the category 
$\mathsf{Cond}$ of condensed sets as $\mathsf{Cond} = \varinjlim_\kappa 
\Cond_\kappa$. Similarly, define the category $\mathsf{CondAb}$ of 
condensed abelian groups as $\mathsf{CondAb} = \varinjlim_\kappa 
\CondAb_\kappa$. In practice, we will slightly abuse notation and 
think of condensed sets or condensed abelian groups as contravariant 
functors defined on the entire category $\mathsf{Prof}$.
Given a condensed set/abelian group $T$, we call $T(\ast)$ the 
\emph{underlying set/abelian group} of $T$. We note that an element 
$T$ of $\Cond$ or $\CondAb$ is fully determined by its restriction 
to $\mathsf{ED}$ (cf.\ \cite[Proposition 2.7]{CS1}). 

Given a Hausdorff space $X$, one can define a condensed set 
$\ul{X}$ by letting $\ul{X}(S) = C(S,X)$, i.e., $X(S)$ is the set of 
continuous functions from $S$ to $X$ for all $S \in \mathsf{Prof}$; we will sometimes term this $\ul{X}$ the \emph{condensation} of $X$. 
This operation defines a functor from the category $\mathsf{Haus}$ of 
Hausdorff topological spaces to $\Cond$. When restricted to the full subcategory of $\mathsf{Haus}$ spanned by the 
compactly generated Hausdorff spaces, it is a fully faithful 
embedding (cf.\ \cite[Proposition 1.7]{CS1}). 

Similarly, given a topological abelian group $A$, one can define 
a condensed abelian group $\ul{A}$ by letting $\ul{A}(S) = 
C(S,A)$ (considered as an abelian group with pointwise addition) 
for all $S \in \mathsf{Prof}$. When restricted to the class 
$\mathsf{LCA}$ of locally compact abelian groups, this again 
describes a fully faithful embedding of $\mathsf{LCA}$ into 
$\mathsf{CondAb}$.

For each uncountable cardinal $\kappa$, there is an obvious 
forgetful functor from $\CondAb_\kappa$ to 
$\Cond_\kappa$, and this functor admits a left adjoint that sends 
a $\kappa$-condensed set $T$ to a $\kappa$-condensed abelian 
group denoted $\Z[T]$. Concretely, as noted in \cite[\S 2]{CS1}, 
$\Z[T]$ is the sheafification of the functor from $\mathsf{Prof}_\kappa$ 
to $\Ab$ that sends each $S \in \mathsf{Prof}_\kappa$ to the free group 
on $T(S)$. Moreover,  $\{\Z[\ul{S}] \mid S \in \mathsf{ED}_\kappa\}$ is, for any strong 
limit cardinal $\kappa$, a set 
of compact projective generators for $\CondAb_\kappa$. 
These facts straightforwardly 
propagate upward to the category $\CondAb$ of all condensed abelian 
groups.

\section{Algebraic preliminaries} \label{prelim_sec}

Given any abelian group $A$, a \emph{free resolution} of $A$ 
is a short exact sequence of the form 
\[
  0 \ra K \xrightarrow{\iota} F \xrightarrow{\pi} A \ra 0
\]
in which both $F$ and $K$ are free abelian groups. In such 
resolutions, we will often implicitly assume that $K$ is a subgroup 
of $F$ and $\iota$ is the inclusion map. We note that every 
infinite abelian group $A$ has a canonical free resolution, in which 
$F$ is the free abelian group generated by the underlying set of $A$ 
and $K$ is the kernel of the induced surjection from $F$ onto $A$. 
The elements of $K$ are therefore the formal linear combinations of 
elements of $F$ that evaluate to $0$ in~$A$.

By applying the functor $\Hom(\cdot, \Z)$ to a free resolution of $A$ and using the fact 
that $F$ and $K$ are free, and therefore Whitehead, we obtain an exact sequence of the form
\[
  0 \ra \Hom(A,\Z) \ra \Hom(F, \Z) \ra \Hom(K, \Z) \ra 
  \Ext^1(A, \Z) \ra 0
\]
We thus see that $A$ is Whitehead if and only if the map $\Hom(F, \Z) \ra \Hom(K, \Z)$ is 
surjective, i.e., if and only if every group homomorphism $\varphi:K \ra \Z$ extends to a homomorphism 
$\psi:F \ra \Z$.

We now turn our attention to the situation in the category of 
condensed abelian groups. $\CondAb$ is a symmetric monoidal category with an internal Hom-functor, which we will denote by 
$\iHom : \CondAb^{\mathrm{op}} \times \CondAb \rightarrow \CondAb$. If $A$ and $B$ are 
sufficiently nice topological groups, then $\iHom(\ul{A}, \ul{B})$ has a particularly nice 
description:

\begin{proposition}[Clausen--Scholze, {\cite[Proposition 4.2]{CS1}}] \label{ihom_prop}
  Suppose that $A$ and $B$ are Hausdorff topological abelian groups and $A$ is compactly generated. 
  Then there is a natural isomorphism of condensed abelian groups
  \[
    \iHom(\ul{A},\ul{B}) \cong \ul{\Hom(A,B)},
  \]
  where, on the right-hand side, $\Hom(A,B)$ is computed in the category of topological abelian 
  groups and endowed with the compact-open topology. 
\end{proposition}

$\iHom$ has corresponding derived functors, the internal Ext-functors 
\[
  \iExt^n : \CondAb^{\mathrm{op}} \times \CondAb \rightarrow \CondAb
\] 
for $n \geq 1$ and, at the level of the derived category, an internal $R\Hom$-functor
\[
  R\iHom : D(\CondAb)^{\mathrm{op}} \times D(\CondAb) \rightarrow D(\CondAb);
\]
The former arise, of course, as the $n^{\mathrm{th}}$ cohomology groups of the latter.
The following lemma shows that these functors behave as expected in relation to their classical analogues in 
$\Ab$; namely, the classical functors can be recovered by evaluating their condensed analogues at 
a point. When we need to be careful about the category in which certain functors are being 
considered, we will write the name of the category as a subscript on the functor, i.e., 
$\Ext^1_{\Ab}$ is the $\Ext$-functor in the category of abelian groups.

\begin{lemma}
\label{lemma:evaluation}
  For any abelian groups $A$ and $B$ (with the discrete topology),
  \[
    \iExt^1_{\CondAb}(\ul{A},\ul{B})(\ast) = \Ext^1_{\Ab}(A,B).
  \]
  More generally, for any condensed abelian groups $T$ and $M$, we have
  \[
    R\iHom_{\CondAb}(T,M)(\ast) = R\Hom_{\CondAb}(T,M),
  \]
  and, for discrete abelian groups $A$ and $B$, we in turn have
  \[
    R\Hom_{\CondAb}(\ul{A},\ul{B}) = R\Hom_{\Ab}(A,B).
  \]
\end{lemma}

\begin{proof}
  As noted in \cite[proof of Corollary 4.8]{CS1} (and as follows from items (ii) and (iii) of its page 13),
  \[
    R\iHom_{\CondAb}(T,M)(S)=R\Hom_{\CondAb}(T\otimes\mathbb{Z}[S],M).
  \]
  for any condensed abelian groups $T,M$ and any extremally disconnected profinite set $S$.
  When $S=\ast$, our general claim follows from the facts that $\Z[*]=\ul{\Z}$ and that 
  $\ul{\Z}$ is the unit object for the condensed tensor product $\otimes$.
To see the assertion for discrete abelian groups, recall that their category fully and faithfully embeds into $\CondAb$, and observe by \cite[Props. 2.18 and 2.19]{Aparicio_condensed_21} that if $\cdots\to P^1\to P^0\to A$ is a projective (hence free) resolution of $A$ then $\cdots\to \underline{P}^1\to \underline{P}^0\to \ul{A}$ is a projective resolution of $\ul{A}$.
Writing $\mathcal{P}$ and $\underline{\mathcal{P}}$ for $\cdots\to P^1\to P^0\to 0$ and $\cdots\to \underline{P}^1\to \underline{P}^0\to 0$, respectively, we then have
\[
    R\Hom_{\CondAb}(\ul{A},\ul{B}) = \Hom_{\CondAb}(\ul{\mathcal{P}},\ul{B}) = \Hom_{\Ab}(\mathcal{P},B) = R\Hom_{\Ab}(A,B),
  \]
  as desired.
\end{proof}

It is natural at this point in the discussion to declare a \emph{condensed} abelian group $T$ to be Whitehead if $\iExt^1(T, \ul{\Z}) = 0$. Observe next that,
in close analogy with the classical situation, 
if $F$ is a free (discrete) abelian group, 
then $\ul{F}$ is Whitehead in $\CondAb$. To see this, 
let $\kappa$ be a cardinal, and let $F = \bigoplus_\kappa 
\Z$ be the free group on $\kappa$ generators; it follows from \cite[Proposition 2.19]{Aparicio_condensed_21} that $\ul{F} \cong \bigoplus_\kappa \ul{\Z}$, where the direct sum on the right is computed in $\CondAb$. By \cite[Remark 3.10 and Lemma~4.2]{Aparicio_condensed_21}, we then have
\[
  R\iHom_{\CondAb}(\ul{F}, \ul{\Z}) = \prod_\kappa 
  R\iHom_{\CondAb}(\ul{\Z},\ul{\Z}) = \prod_\kappa \ul{\Z}[0],
\]
where $\ul{\Z}[0]$ denotes the chain complex that is $\ul{\Z}$ 
in degree $0$ and $0$ elsewhere. In particular, we have 
$\iExt^1_{\CondAb}(\ul{F}, \ul{\Z}) = 0$.

One reasonable translation (though, as we will discuss later, not the only possible one) 
of Whitehead's problem to the condensed setting is then the following:
\begin{quote}
  Suppose that $A$ is a discrete abelian group and $\ul{A}$ is Whitehead in $\CondAb$. Must $A$ 
  be free?
\end{quote}
Note that, by Lemma \ref{lemma:evaluation}, the assertion that $\ul{A}$ is Whitehead in $\CondAb$ 
is a strengthening of the assertion that $A$ is Whitehead in $\Ab$; the latter is equivalent 
to asserting that $\iExt^1(\ul{A},\ul{\Z})(\ast) = 0$, whereas the former is equivalent to 
asserting that $\iExt^1(\ul{A}, \ul{\Z})(S) = 0$ for \emph{every} profinite set $S$. And, indeed, Theorem 
\ref{cs_whitehead_thm} implies that, unlike the classical Whitehead problem, this translation to 
the condensed setting has a positive answer in $\ZFC$.

In the coming sections, we will give a more concrete, combinatorial proof of Theorem \ref{cs_whitehead_thm}. 
In particular, given a nonfree discrete abelian group $A$, we will identify a natural $S \in \mathsf{ED}$ such that $\iExt^1(\ul{A},\ul{\Z})(S) \neq 0$. The process by which this will be done is 
analogous to that described in the classical setting at the beginning of this section, so in the interest of 
symmetry, we end this section by giving an outline of the structure of the argument.

Suppose that $A$ is a nonfree discrete abelian group, and let $0 \ra K \ra F \ra A \ra 0$ be a
free resolution of $A$. Using the discreteness of the groups under consideration, it is straightforward to observe that the induced sequence 
$0 \ra \ul{K} \ra \ul{F} \ra \ul{A} \ra 0$ is exact in $\CondAb$. 
Applying $\iHom(\cdot, \ul{\Z})$ to this short exact sequence yields the exact sequence
\[
  0 \ra \iHom(\ul{A}, \ul{\Z}) \ra \iHom(\ul{F}, \ul{\Z}) \ra \iHom(\ul{K}, \ul{\Z}) 
  \ra \iExt^1(\ul{A}, \ul{\Z}) \ra 0.
\]
Given an $S \in \mathsf{ED}$, the functor from $\CondAb$ to $\Ab$ sending 
$T \in \CondAb$ to $T(S) \in \Ab$ is exact (cf.\ \cite[Lemma 6.1.3]{le_stum}).
Thus, evaluating the above sequence at a space $S \in \mathsf{ED}$ and recalling Proposition \ref{ihom_prop}, we obtain
\[
  0 \ra \ul{\Hom(A,\Z)}(S) \ra \ul{\Hom(F,\Z)}(S) \ra \ul{\Hom(K,\Z)}(S) \ra 
  \iExt^1(\ul{A}, \ul{\Z})(S) \ra 0.
\]
It follows that $\iExt^1(\ul{A}, \ul{\Z})(S) = 0$ if and only if the 
map 
\[
  \ul{\Hom(F,\Z)}(S) \ra \ul{\Hom(K,\Z)}(S)
\]
is a surjection. Recall that 
$\ul{\Hom(F,\Z)}(S)$ consists of all continuous functions $\psi:S \ra \Hom(F,\Z)$ (and similarly 
for $K$), where $\Hom(F,\Z)$ is given the compact-open topology (which, since $F$ and $\Z$ are 
discrete, is simply the topology $\Hom(F,\Z)$ inherits as a closed subspace of 
$\prod_F \Z$ endowed with the product topology). Moreover, the map
$\ul{\Hom(F,\Z)}(S) \ra \ul{\Hom(K,\Z)}(S)$ sends a continuous $\psi:S \ra \Hom(F,\Z)$ to 
the map $\phi:S \ra \Hom(K,\Z)$ defined by letting $\phi(s) = \psi(s) \restriction K$ 
for all $s \in S$. 

Therefore, to show that $\ul{A}$ is not Whitehead, it will suffice to find 
a space $S \in \mathsf{ED}$ and a continuous $\varphi:S \ra \Hom(K,\Z)$ that does not ``lift pointwise" 
to a map from $S$ to $\Hom(F, \Z)$, i.e., for which there is no continuous $\psi:S \ra 
\Hom(F,\Z)$ for which $\psi(s) \restriction K = \varphi(s)$ for all $s \in S$.
The construction of such an $S$ and $\varphi$ will be the subject of the next sections. For expository reasons, and because it more closely derives from 
classical arguments around Whitehead's problem, we begin by finding $S$ and 
$\varphi$ as above but for which $S$ is merely profinite and not extremally 
disconnected. We then show how to improve these arguments to obtain 
$S \in \mathsf{ED}$, as desired.

\section{Countable groups} \label{countable_sec}

Recall that, if $A$ is a torsion-free abelian group and $X$ is a subgroup of $A$, we say that 
$X$ is a \emph{pure} subgroup of $A$ if $A/X$ is torsion-free. Given any subgroup $X$ of $A$, 
the \emph{pure closure} of $X$ in $A$ is the smallest pure subgroup of $A$ containing $X$; 
concretely, this can be seen to be equal to $\{a \in A \mid \exists n > 0 ~ [ na \in X]\}$.
The following theorem gives a useful characterization of countable free abelian groups (recall that a finitely generated abelian group is 
free if and only if it is torsion-free).

\begin{theorem} [\cite{pontrjagin} (cf.\ {\cite[Theorem 4.2]{eklof}})] \label{thm:pontrjagin}
  Suppose that $A$ is a countable torsion-free abelian group. Then $A$ is free if and only 
  if every finitely generated subgroup of $A$ is contained in a finitely generated pure 
  subgroup of $A$.
\end{theorem}

As mentioned above, Stein \cite{stein} proved that every countable Whitehead group is free. The 
following theorem gives one way of seeing this (and slightly more).

\begin{theorem} \label{ctbl_thm}
  Suppose that $A$ is a countable, torsion-free, nonfree abelian group,
  \[
    0 \ra K \xrightarrow{\subseteq} F \xrightarrow{\pi} A \ra 0
  \]
  is a free resolution of $A$, and $B_K$ is a basis for $K$. Then there is a~homomorphism $\varphi : K \ra \Z$ such that
  \begin{enumerate}
    \item $\varphi$ does not extend to a homomorphism $\psi:F \ra \Z$; and
    \item $\varphi[B_K] \subseteq \{0,1\}$.
  \end{enumerate}
\end{theorem}

\begin{proof}
  First, fix any basis $B_F$ for the free group $F$ and notice that $B_K$ is infinite: if not, then $K$ would be included in a~subgroup $G$ of $F$ generated by a~finite $S\subset B_F$; then $F = G \oplus H$ where $H$ is generated by $B_F\setminus S$, and we get $A \cong F/K \cong (G/K) \oplus H$ where $H$ is free and $G/K$ is finitely generated torsion-free, hence free, in contradiction with the nonfreeness of $A$.

  Next, let us argue that, without loss of generality, we may assume 
  that $F$ (and hence also $K$) is countable. To see this, suppose that 
  $K$ is uncountable. Let $\chi$ be a sufficiently large regular cardinal, 
  and let $N$ be a countable elementary submodel of 
  $(H(\chi), \in, A, B_F, B_K)$. Let $F_0 := F \cap N$ and 
  $K_0 := K \cap N$. By elementarity, $F_0$ and $K_0$ are free, with 
  bases $B_F \cap N$ and $B_K \cap N$, respectively, and 
  \[
    0 \ra K_0 \xrightarrow{\subseteq} F_0 \xrightarrow{\pi \restriction 
    F_0} A \ra 0
  \]
  is a free resolution of $A$ by countable free groups. Suppose that 
  we can find a~homomorphism $\varphi_0:K_0 \ra \Z$ such that 
  $\varphi_0[B_K \cap N] \subseteq \{0,1\}$ and $\varphi_0$ does 
  not extend to a homomorphism $\psi_0 : F_0 \ra \Z$. Then we can 
  extend $\varphi_0$ to a~homomorphism $\varphi:K \ra \Z$ by setting 
  $\varphi(t) = 0$ for all $t \in B_K \setminus N$. The fact that 
  $\varphi_0$ does not extend to a~homomorphism from $F_0$ to $\Z$ 
  implies that $\varphi$ does not extend to a~homomorphism from $F$ 
  to $\Z$. Therefore, we may assume that $F$ is countable.
  
  We can thus injectively enumerate $B_K$ as $\langle t_n \mid n < \omega \rangle$ and, for each $n<\omega$, let $K_n$ denote the subgroup of $K$ generated by $\{t_k \mid k<n\}$. We will define our homomorphism $\varphi:K \ra \Z$ by recursion, requiring that $\varphi(t_n) \in \{0,1\}$ for all $n < \omega$. We will do this by constructing a~strictly increasing sequence $\langle n_i \mid i < \omega \rangle$ 
  of natural numbers and, for each $i < \omega$, a~homomorphism $\varphi_i:K_{n_i} \rightarrow 
  \Z$ in such a way that $\varphi_i[\{t_k \mid k<n_i\}]\subseteq\{0,1\}$ and, for all $i < j < \omega$, $\varphi_j\restriction K_{n_i} = \varphi_i$. In the process, we will also 
  define a $\subseteq$-increasing sequence $\langle Y_n \mid n < \omega \rangle$ of 
  finite subsets of $B_F$ such that, for all $i < \omega$, letting $F_i$ denote 
  the subgroup of $F$ generated by $Y_i$, we have $K_{n_i} \subseteq F_i$.
  
  By Theorem \ref{thm:pontrjagin}, since $A$ is countable, torsion-free, and nonfree, there is a finitely generated subgroup $X_0$ of $A$ such that its pure closure $X$ is not finitely generated. Let $Y_0$ be a~finite subset of $B_F$ such that $X_0\subseteq \pi[F_0]$ where $F_0$ denotes the subgroup of $F$ generated by $Y_0$.	Put $n_0 = 0$ and enumerate $\Hom(F_0,\Z) = \{\psi_i \mid i<\omega\}$. Finally, let $\varphi_0$ be the zero homomorphism.
	
	Assume that we have already defined $Y_i$, $F_i$, $n_i$ and $\varphi_i:K_{n_i} \rightarrow \Z$, for some $i<\omega$, in such a way that $K_{n_i}\subseteq F_i$. Consider any $z_i\in F$ such that $x_i:=\pi(z_i)\in X\setminus \pi[F_i]$.
	
	Since $X$ is the pure closure of $X_0$, there is a~positive integer $j_i>1$ such that $j_ix_i = \pi(j_iz_i)\in X_0\subseteq \pi[F_0]$. It follows that there exists $y_i\in F_0$ such that $\pi(y_i) = \pi(j_iz_i)$, whence $s_i:= j_iz_i-y_i\in K = \Ker(\pi)$. We can thus express $s_i$ as $a_i + \sum_{n_i\leq k<n_{i+1}} b_kt_k$ where $a_i\in K_{n_i}$, $n_i< n_{i+1}<\omega$ and $b_k\in \Z$ for each $k\in [n_i,n_{i+1})$. Let $Y_{i+1}$ be any finite subset of $B_F$ such that $Y_i\subseteq Y_{i+1}$ and $K_{n_{i+1}} \subseteq F_{i+1}$ where $F_{i+1}$ denotes the subgroup of $F$ generated by $Y_{i+1}$.
	
	\begin{claim}
    There is $k \in [n_i, n_{i+1})$ such that $j_i \nmid b_k$.
  \end{claim}  
  
  \begin{proof}
    Suppose not, and set
    \[
      s_i^* := s_i - a_i = \sum_{k = n_i}^{n_{i+1}-1} b_k t_k.
    \]
    Then $s_i^* \in K$ is of the form $j_iz_i - r$, where $r \in F_i$. Moreover, since all of the coefficients $b_k$ for $k \in [n_i, n_{i+1})$ are divisible by $j_i$, it follows that $s_i^*$ is divisible by $j_i$ (in $K$). Since $r = j_i z_i - s_i^*$ is in $F_i$ and $F_i$ 
    is generated by a subset of $B_F$, this means that there exists $r^*\in F_i$ such that $r = j_i r^*$. We get $s_i^* = j_i(z_i-r^*)\in K$ which implies $z_i - r^* \in K$ using the fact that $A$ is torsion-free. In particular, $\pi(z_i) = \pi(r^*)$, so $\pi(z_i) \in \pi[F_i]$, contradicting the fact that $\pi(z_i) = x_i \notin \pi[F_i]$.
  \end{proof}
  
  Fix a $k_i \in [n_i, n_{i+1})$ as given by the claim. To construct $\varphi_{i+1}$, it suffices to define $\varphi_{i+1}(t_k)$ for all $k \in [n_i, n_{i+1})$. 
  For $k \in [n_i, n_{i+1}) \setminus \{k_i\}$, simply let $\varphi_{i+1}(t_k) = 0$. To define $\varphi_{i+1}(t_{k_i})$, consider the homomorphisms $\psi_i:F_0 \rightarrow \Z$ and $\varphi_i:K_{n_i} \rightarrow \Z$ and the number
  \[
    d_i := \varphi_i(a_i) + \psi_i(y_i).
  \]
  If $d_i$ is divisible by $j_i$, then let $\varphi_{i+1}(t_{k_i}) = 1$. Otherwise, let $\varphi_{i+1}(t_{k_i}) = 0$.
  
  Now let $\varphi:=\bigcup_{i<\omega}\varphi_i$. We claim that this is as desired. Clearly, $\varphi : K \ra \Z$ is a homomorphism and $\varphi[B_K] \subseteq \{0,1\}$. We must therefore show that every homomorphism $\psi:F \ra \Z$ fails to extend $\varphi$. 
  
  To this end, fix a homomorphism $\psi:F \ra \Z$. Find $i < \omega$ such that $\psi \restriction 
  F_0 = \psi_i$. We claim that $\psi(s_i) \neq \varphi(s_i)$. Note that
  \[
    \psi(s_i) = j_i \psi(z_i) - \psi_i(y_i),
  \]
  and
  \[
    \varphi(s_i) = b_{k_i} \varphi_{i+1}(t_{k_i}) + \varphi_i(a_i).
  \]
  Therefore, if it were the case that $\psi(s_i) = \varphi(s_i)$, we would have
  \begin{align*}
    j_i \psi(z_i) &= b_{k_i} \varphi_{i+1}(t_{k_i}) + \varphi_i(a_i) 
    + \psi_i(y_i) \\ &= b_{k_i} \varphi_{i+1}(t_{k_i}) + d_i.
  \end{align*}
  Since $j_i \psi(z_i)$ is clearly divisible by $j_i$, it follows that the right hand side of this equation must also be divisible by $j_i$.
  Suppose first that $d_i$ is divisible by $j_i$. Then $\varphi_{i+1}(t_{k_i}) = 1$, so, since $k_i$ was chosen so that $b_{k_i}$ is not divisible by $j_i$, it follows that the right hand side of the above equation is not divisible by $j_i$, which is a~contradiction. 
  If, on the other hand, $d_i$ is not divisible by $j_i$, then $\varphi_{i+1}(t_{k_i}) = 0$, so the right hand side of the above equation equals $d_i$ and is hence not divisible by $j_i$, which is again a~contradiction. Therefore, it must be the case that $\psi(s_i) \neq \varphi(s_i)$.
\end{proof}

It is also well-known that all Whitehead groups are torsion-free. We now record a more general result, analogous to Theorem~\ref{ctbl_thm}, implying this fact.

 \begin{proposition} \label{prop:torsion}
   Let
   \[
     0 \ra K \xrightarrow{\subseteq} F \xrightarrow{\pi} A \ra 0
   \]
   be a short exact sequence of groups where $F$ is torsion-free and $K$ is free with a~basis~$B_K$. Assume that $A$ is not torsion-free (i.e.\ that $K$ is not a pure subgroup of $F$). Then there is a~homomorphism $\varphi:K \ra \Z$ such that
   \begin{enumerate}
     \item $\varphi$ does not extend to a homomorphism $\psi:F \ra \Z$; and
     \item there is a unique $t^* \in B_K$ such that
     \begin{itemize}
       \item $\varphi(t^*) = 1$;
       \item $\varphi(t) = 0$ for all $t \in B_K \setminus \{t^*\}$.
     \end{itemize}
   \end{enumerate}
 \end{proposition}

\begin{proof}
Being free, $K$ is isomorphic to the direct sum of copies of $\Z$ indexed by~$B_K$. We consider the canonical embedding $\nu : K \to \Z^{B_K}$, i.e.,\ for every $t,s\in B_K$, $\nu(t)(t) = 1$ and $\nu(t)(s) = 0$ if $s\neq t$. Then $\nu$ is pure, i.e.,\ $\nu[K]$ is pure in $\Z^{B_K}$ (in fact, $\nu$ is even an elementary embedding). Since $K$ is not pure in $F$, it follows that there does not exist any homomorphism $\mu: F \to \Z^{B_K}$ such that $\mu\restriction K = \nu$.

For each $t\in B_K$, let us denote by $p_t:\Z^{B_K} \to \Z$ the $t$-th canonical projection. Since any homomorphism into a direct product is uniquely determined by specifying its compositions with all the canonical projections, there has to be some $t^*\in B_K$ such that $\varphi:=p_{t^*}\nu$ cannot be extended to a homomorphism from $F$ into $\Z$. One readily checks that $t^*$ and $\varphi$ possess the desired properties.
\end{proof}

\begin{remark}
    In the formulation of Theorem~\ref{ctbl_thm}, we could not have hoped to find $\varphi$ with $(\varphi\restriction B_K)^{-1}\{1\}$ finite, let alone a~$\varphi$ satisfying the stronger condition (2) from Proposition~\ref{prop:torsion}. The reason is that any given countable torsion-free group~$A$, being the union of a~countable chain of finitely generated free subgroups, is isomorphic to a~(pure) subgroup of $\Z^{\omega}/\Z^{(\omega)}$, cf.\ \cite[Proposition~3.1]{Saroch} or \cite[Theorem~3.3.2]{Prest}. Here, $\Z^{(\omega)}$ denotes the direct sum of countably many copies of $\Z$, i.e., the subgroup of $\Z^\omega$ consisting of elements with finite support.
    
    Let $B = \{e_n \mid n<\omega\}$ denote the canonical basis of the free group $\Z^{(\omega)}$. Then for $A$ as above, there exists a subgroup $F$ of $\Z^{\omega}$ such that the following diagram with exact rows, where $\mu$ and $\nu$ denote the respective identity embeddings, commutes.
    \[\xymatrix{0 \ar[r] & \Z^{(\omega)} \ar[r]^-{\nu} & \Z^{\omega} \ar[r] & \Z^{\omega}/\Z^{(\omega)} \ar[r] & 0 \\ 0 \ar[r] & \Z^{(\omega)} \ar[r]^-{\subseteq} \ar@{=}[u] & F \ar[r] \ar[u]^-{\mu} & A \ar[r] \ar[u] & 0
    }\]
    
    It follows that $F$ is countable and thus free since all countable subgroups of $\Z^\omega$ are; see \cite[Theorems~IV.2.3 and IV.2.8]{Eklof_Mekler}. Since $\mu$ extends the map $\nu$, we see, in particular, that $p\mu$ extends the map $p\nu$ for any $p\in\Hom(\Z^{\omega},\Z)$. This applies specifically for any $p$ which is a~$\Z$-linear combination of the canonical projections $p_n:\Z^\omega \to \Z$, $n < \omega$. It remains to notice that any $\varphi\in\Hom(\Z^{(\omega)},\Z)$ which is nonzero only on finitely many elements from $B$ is of the form $p\nu$ for $p$ as above.

    On the other hand, since $\Z$ is a~slender group (\cite[Corollary~III.2.4]{Eklof_Mekler}), no homomorphism $\varphi:\Z^{(\omega)}\to \Z$ which is nonzero on infinitely many elements from $B$ can be extended to a~homomorphism from $\Z^\omega$ into $\Z$. Of course, it can happen that a~particular such $\varphi$ can be extended to a~homomorphism from $F$ into $\Z$. 
\end{remark}

\section{Uncountable groups} \label{unctble_sec}

In what follows, given a cardinal $\kappa$, $2^\kappa$ denotes the topological space consisting 
of all functions $f:\kappa \ra 2$, equipped with the product topology. In particular, basic clopen 
subsets of $2^\kappa$ are of the form $U_\sigma := \{f \in 2^\kappa \mid \sigma \subseteq f\}$, where 
$\sigma$ is a finite partial function from $\kappa$ to $2$. Note that $2^\kappa$ is a profinite set.

\begin{theorem} \label{thm:main_thm_1}
  Suppose that $A$ is a nonfree abelian group, and let 
  $\kappa$ be the minimal cardinality of a nonfree subgroup of $A$. 
  Suppose moreover that $0 \ra K \xrightarrow{\subseteq} F \ra A \ra 0$ is a free resolution of $A$. Then there 
  is a continuous map $\varphi:2^\kappa \ra \Hom(K, \Z)$ that cannot be lifted pointwise to a continuous 
  map $\psi:2^\kappa \ra \Hom(F, \Z)$.
  
  Moreover, given a basis $B_K$ for $K$, we can require that, for all $f \in 2^\kappa$ and all 
  $t \in B_K$, we have $\varphi(f)(t) \in \{0,1\}$.
\end{theorem}

\begin{proof}
  If $A$ has nonzero torsion elements, then Proposition \ref{prop:torsion} provides us 
  with a homomorphism $\tau \in \Hom(K,\Z)$ such that $\tau[B_K] \subseteq \{0,1\}$ 
  and $\tau$ cannot be lifted to an element of $\Hom(F,\Z)$. We can in this case let 
  $\varphi:2^\kappa \ra \Hom(K,\Z)$ be the constant map taking value $\tau$. Thus, assume 
  for the rest of the proof that $A$ is torsion-free.

  The proof will be by induction on $\kappa$. Note that, by Shelah's singular compactness theorem \cite{shelah_singular_compactness}, $\kappa$ must 
  be regular.  Let $B_F$ and $B_K$ be 
  bases for $F$ and $K$, respectively, and note that $|B_K| \leq |B_F|$. Let $\chi$ be a~sufficiently large regular cardinal, so that all objects of interest are in $H(\chi)$, and let 
  $\vartriangleleft$ be a fixed well-ordering of $H(\chi)$.
  
  We first claim that we may assume both that $A$ is \emph{almost free}, 
  i.e., that $\kappa = |A|$, and that $|F| = \kappa$. To see this, let $N$ be an elementary submodel of $(H(\chi),\in, \vartriangleleft,$ $ A,B_F,B_K)$ 
  with $|N| = \kappa \subseteq N$. Let $A_0 := A \cap N$. By elementarity, $N$ contains as 
  an element a nonfree subgroup of $A$ of cardinality $\kappa$. Since $\kappa \subseteq N$, 
  this subgroup is a subset of $A_0$; therefore, $A_0$ is nonfree, and, by the minimality of 
  $\kappa$, almost free. Let $F_0 := F \cap N$ and $K_0 := K \cap N$; by elementarity, 
  $B_{F,0} := B_F \cap N$ and $B_{K,0} := B_K \cap N$ are bases for $F_0$ and $K_0$, 
  respectively, and $0 \ra K_0 \ra F_0 \ra A_0 \ra 0$ is a free resolution of $A_0$. Moreover, we clearly have $|F_0| = \kappa$.
  
  Suppose that we can find a continuous $\varphi_0 : 2^{\kappa} \ra 
  \Hom(K_0,\Z)$ that does not lift pointwise to a continuous map $\psi_0 : 2^{\kappa} \ra 
  \Hom(F_0, \Z)$, and assume moreover that, for all $f \in 2^{\kappa}$ and all 
  $t \in B_{K,0}$, we have $\varphi_0(f)(t) \in \{0,1\}$. 
  Then $\varphi_0$ can be extended pointwise to a continuous map 
  $\varphi : 2^\kappa \ra \Hom(K,\Z)$ by setting 
  \[
    \varphi(f)(t) = \begin{cases}
      \varphi_0(f)(t) & \text{if } t\in B_{K,0} \\
      0 & \text{otherwise}
    \end{cases}
  \]
  for all $t \in B_K$ and extending linearly. The continuity of $\varphi$ follows from the continuity of $\varphi_0$, and our construction 
  ensures that $\varphi(f)(t) \in \{0,1\}$ for all $f \in 2^\kappa$ and all $t \in B_K$. Moreover, $\varphi$ does not lift pointwise to 
  a continuous $\psi: 2^\kappa \ra \Hom(F,\Z)$, as pointwise restriction 
  of such a $\psi$ to $F_0$ would induce a continuous pointwise 
  extension $\psi_0:2^\kappa \ra \Hom(F_0, \Z)$ of $\varphi_0$.

  Thus, for the rest of the proof we will assume that $A$ is almost free 
  and $|F| = \kappa$. The case of $\kappa = \aleph_0$ is covered by 
  Theorem \ref{ctbl_thm}, so we may assume that $\kappa > \aleph_0$.
  Let $\vec{M} = \langle M_\alpha \mid \alpha < \kappa \rangle$ be a $\in$-increasing, continuous sequence of 
  elementary submodels of $(H(\chi), \in, \vartriangleleft, A,F,K)$ such that 
  $|M_\alpha| < \kappa$ and $\delta_\alpha := M_\alpha \cap \kappa \in \kappa$ for all $\alpha < 
  \kappa$. For each $\alpha < \kappa$, let $A_\alpha := A \cap M_\alpha$, $F_\alpha := F \cap 
  M_\alpha$, and $K_\alpha := K \cap M_\alpha$. Since $A$, $F$, and $K$ each have size at most 
  $\kappa$ and $\vec{M}$ is $\in$-increasing, we have $A = \bigcup_{\alpha < \kappa} A_\alpha$, 
  $F = \bigcup_{\alpha < \kappa} F_\alpha$, and $K = \bigcup_{\alpha < \kappa} K_\alpha$.

  Let $E$ be the set of $\alpha < \kappa$ such that there exists 
  a $\beta$ such that $\alpha < \beta < \kappa$ and $A_\beta / A_\alpha$ 
  is nonfree. Note that, if $\alpha < \beta < \beta' < \kappa$ and 
  $A_\beta/A_\alpha$ is nonfree, then also $A_{\beta'}/A_\alpha$ 
  is nonfree, since it contains $A_\beta/A_\alpha$ as a subgroup.
  We claim that $E$ is stationary. If not, then there exists a club 
  $D \subseteq \kappa$ such that, for all $\alpha < \beta$, both 
  in $D$, the group $A_\beta/A_\alpha$ is free. Moreover, since 
  $A$ is almost free, we know that $A_{\min(D)}$ is free. Therefore, 
  it is straightforward to build a basis for $A$ by recursion 
  along $D$ (cf.\ \cite[Theorem 5.3]{eklof}), contradicting the assumption 
  that $A$ is not free.

  For all $\alpha \in E$, let $\alpha < \beta_\alpha < \kappa$ be 
  such that $A_{\beta_\alpha}/A_\alpha$ is not free. Let 
  $D := \{\gamma < \kappa \mid \forall \alpha < \gamma \;[\beta_\alpha < 
  \gamma]\}$. Then $D$ is a club in $\kappa$. Letting 
  $\langle \gamma_\eta \mid \eta < \kappa \rangle$ be the increasing 
  enumeration of $D$, the following two facts are immediate:
  \begin{enumerate}
      \item $\{\eta < \kappa \mid \gamma_\eta \in E\}$ is stationary 
      in $\kappa$;
      \item for all $\eta < \kappa$ for which $\gamma_\eta \in E$, 
      the set $A_{\gamma_{\eta+1}}/A_{\gamma_\eta}$ is not free.
  \end{enumerate}
  Therefore, by thinning out $\vec{M}$ to only include those models 
  indexed by elements of $D$ and then reindexing, we may 
  assume that the set
  \[
    S := \{\alpha < \kappa \mid A_{\alpha+1}/A_\alpha \text{ is nonfree}\}
  \]
  is stationary in $\kappa$.
  
  For each $\alpha \in S$, let $A_\alpha^* := A_{\alpha+1}/A_\alpha$, $F_\alpha^* := 
  F_{\alpha+1}/F_\alpha$, and $K_\alpha^* := K_{\alpha+1}/K_\alpha$. Similarly, let 
  $B_{F,\alpha}^* := B_F \cap (M_{\alpha+1} \setminus M_\alpha)$ and $B_{K,\alpha}^* := 
  B_K \cap (M_{\alpha+1} \setminus M_\alpha)$. Note that each $F_\alpha^*$ and $K_\alpha^*$ is free, with bases given by $\{z + F_\alpha \mid z \in 
  B_{F,\alpha}^*\}$ and $\{t + K_\alpha \mid t \in B_{K,\alpha}^*\}$, respectively. Also, 
  our free resolution of $A$ induces a free resolution
  \[
    0 \ra K_\alpha^* \ra F_\alpha^* \ra A_\alpha^* \ra 0
  \]
  of $A_\alpha^*$. Since $\alpha \in S$, we know that $A_\alpha^*$ is nonfree. Let $\kappa_\alpha := 
  |A_\alpha^*| < \kappa$, and apply the inductive hypothesis to find a continuous 
  $\varphi_\alpha : 2^{\kappa_\alpha} \ra \Hom(K_\alpha^*,\Z)$ that does not lift pointwise 
  to a continuous $\psi_\alpha : 2^{\kappa_\alpha} \ra \Hom(F_\alpha^*, \Z)$. Note that we 
  are slightly abusing terminology here, since $K_\alpha^*$ is not formally a subgroup of 
  $F_\alpha^*$. Nonetheless, we identify $K_\alpha^*$ with the subgroup $\{r + F_\alpha \mid 
  r \in K_{\alpha+1}\}$, so our assumption is formally that there is no continuous 
  $\psi_\alpha:2^{\kappa_\alpha} \ra \Hom(F_\alpha^*,\Z)$ such that, for all 
  $f \in 2^{\kappa_\alpha}$ and all $r \in K_{\alpha+1}$, we have $\psi_\alpha(f)(r + F_\alpha) = 
  \varphi_\alpha(f)(r + K_\alpha)$.
  
  We are now ready to describe the construction of our desired continuous map $\varphi:2^\kappa 
  \rightarrow \Hom(K,\Z)$. We first need a bit of notation: given $f \in 2^\kappa$ and 
  $\alpha \in S$, define a map $f_\alpha \in 2^{\kappa_\alpha}$ by letting $f_\alpha(\eta) = 
  f(\alpha + 1 + \eta)$ for all $\eta < \kappa_\alpha$. Now, to define $\varphi$, it suffices 
  to specify $\varphi(f)(t)$ for all $f \in 2^\kappa$ and $t \in B_K$, and then extend each 
  $\varphi(f)$ linearly to all of $K$. To do this, suppose we are given $f \in 2^\kappa$ 
  and $t \in B_K$. If there is $\alpha \in S$ such that $f(\alpha) = 1$ and $t \in B^*_{K,\alpha}$,\footnote{Note that such an $\alpha$ must be unique, if it exists.} then let
  \[
    \varphi(f)(t) = \varphi_\alpha(f_\alpha)(t + K_\alpha).
  \]
  In all other cases, let $\varphi(f)(t) = 0$.
  
  The continuity of $\varphi:2^\kappa \ra \Hom(K,\Z)$ follows immediately from the continuity of 
  $\varphi_\alpha$ for each $\alpha \in S$. Moreover, our construction, and the analogous fact 
  about each $\varphi_\alpha$, ensures that $\varphi(f)(t) \in \{0,1\}$ for each 
  $f \in 2^\kappa$ and $t \in B_K$. It thus remains to show that $\varphi$ cannot be lifted pointwise 
  to a continuous $\psi:2^\kappa \ra \Hom(F,\Z)$.
  
  Suppose for the sake of contradiction that $\psi:2^\kappa \ra \Hom(F,\Z)$ is a continuous map 
  such that $\psi(f) \restriction K = \varphi(f)$ for all $f \in 2^\kappa$. 
  Let $\{z_\gamma \mid \gamma < \kappa\}$ be the $\vartriangleleft$-least injective enumeration of 
  $B_F$, and note that, for all $\alpha < \kappa$, 
  $\{z_\gamma \mid \gamma < \delta_\alpha\}$ is a basis for $F_\alpha$. Temporarily fix 
  $\gamma < \kappa$. By the continuity of $\psi$, for every $f \in 2^\kappa$ there exists a finite 
  $v(f,\gamma) \subseteq \kappa$ such that $\psi(g)(z_\gamma) = \psi(f)(z_\gamma)$ for all 
  $g \in U_{f \restriction v(f,\gamma)}$. By the compactness of $2^\kappa$, there is a finite 
  $\mc F_\gamma \subseteq 2^\kappa$ such that $2^\kappa = \bigcup_{f \in \mc F_\gamma} U_{f 
  \restriction v(f,\gamma)}$. Then, letting $v(\gamma) := \bigcup_{f \in \mc F_\gamma} v(f,\gamma)$, 
  it follows that, for all $f,g \in 2^\kappa$, if $f \restriction v(\gamma) = g \restriction v(\gamma)$, 
  then $\varphi(f)(z_\gamma) = \varphi(g)(z_\gamma)$.
  
  By the stationarity of $S$ and the fact that $\kappa$ is a regular, uncountable cardinal, we can 
  fix $\alpha \in S$ such that $\delta_\alpha = \alpha$ and $v(\gamma) \subseteq \alpha$ for all 
  $\gamma < \alpha$. We will complete the proof by constructing a continuous map 
  $\psi_\alpha : 2^{\kappa_\alpha} \rightarrow \Hom(F^*_\alpha, \Z)$ that lifts $\varphi_\alpha$ 
  pointwise, contradicting our choice of $\varphi_\alpha$.
  
  We first need a bit more notation. Given a function $g \in 2^{\kappa_\alpha}$ and an $i < 2$, define 
  a function $g^+_i \in 2^\kappa$ as follows. First, let $g^+_i(\alpha) = i$. Then, for all 
  $\eta < \kappa_\alpha$, let $g^+_i(\alpha + 1 + \eta) = g(\eta)$. Finally, let 
  $g^+_i(\beta) = 0$ for all $\beta \in \kappa \setminus [\alpha, \alpha + \kappa_\alpha)$.
  Note that, for $i < 2$, we have $(g^+_i)_\alpha = g$.
  
  As usual, to define $\psi_\alpha$, it suffices to specify $\psi_\alpha(g)(z + F_\alpha)$ for 
  all $g \in 2^{\kappa_\alpha}$ and $z \in B^*_{F,\alpha}$. To this end, given such $g$ and $z$, let 
  \[
    \psi_\alpha(g)(z + F_\alpha) = \psi(g^+_1)(z) - \psi(g^+_0)(z).
  \]
  
  We will be done if we show that $\psi_\alpha(g)(t + F_\alpha) = \varphi_\alpha(g)(t + K_\alpha)$ 
  for all $g \in 2^{\kappa_\alpha}$ and $t \in B^*_{K,\alpha}$. To this end, fix such $g$ and $t$. 
  Let $B_{F,\alpha} := \{z_\gamma \mid \gamma < \alpha\} = B_F \cap M_\alpha$, and note that 
  $B_{F,\alpha} \cup B^*_{F,\alpha}$ is a basis for $F_{\alpha+1}$.
  Since $t \in K_{\alpha+1} \subseteq F_{\alpha+1}$, we can find finite subsets $I^- \subseteq B_{F,\alpha}$ 
  and $I^+ \subseteq B^*_{F,\alpha}$ together with integer coefficients $\{a_z \mid z \in I^- \cup 
  I^+\}$ such that, in $F_{\alpha+1}$, we have
  \[
    t = \sum_{z \in I^- \cup I^+} a_z z.
  \]
  Let
  \[
    t^- := \sum_{z \in I^-} a_z z \text{\ \ and \ } t^+ := \sum_{z \in I^+} a_z z.
  \]
  The following observations are immediate, but crucial:
  \begin{itemize}
    \item $t = t^- + t^+$;
    \item $t^- \in F_\alpha$;
    \item $t^+$ is a linear combination of elements of $B^*_{F,\alpha}$; in consequence, we have
    $\psi_\alpha(g)(t^+ + F_\alpha) = \psi(g_1^+)(t^+) - \psi(g_0^+)(t^+)$.
  \end{itemize}
  Since $g_1^+ \restriction \alpha = g_0^+ \restriction \alpha$, our choice of $\alpha$ ensures that 
  $\psi(g_1^+) \restriction F_\alpha = \psi(g_0^+) \restriction F_\alpha$, and hence in particular 
  $\psi(g_1^+)(t^-) = \psi(g_0^+)(t^-)$. But now, putting all of this together we obtain
  \begin{align*}
    \psi_\alpha(g)(t + F_\alpha) &= \psi_\alpha(g)(t^+ + F_\alpha) \\ 
    &= \psi(g_1^+)(t^+) - \psi(g_0^+)(t^+) \\
    &= \psi(g_1^+)(t^-) + \psi(g_1^+)(t^+) - (\psi(g_0^+)(t^-) + \psi(g_0^+)(t^+)) \\ 
    &= \psi(g_1^+)(t) - \psi(g_0^+)(t) \\ 
    &= \varphi(g_1^+)(t) - \varphi(g_0^+)(t) \\ 
    &= \varphi_\alpha(g)(t + K_\alpha) - 0 = \varphi_\alpha(g)(t + K_\alpha).
  \end{align*}
  In the above equation, the first equality holds because $t = t^- + t^+$ and $t^- \in F_\alpha$, 
  and hence $t + F_\alpha = t^+ + F_\alpha$. The second follows from the definition of 
  $\psi$, the third from the observation above that $\psi(g_1^+)(t^-) = \psi(g_0^+)(t^-)$, 
  the fourth by linearity, the fifth from the assumption that $\psi$ lifts $\varphi$ pointwise, 
  and the sixth from the definition of $\varphi$, noting that $g^+_i(\alpha) = i$ for $i < 2$.
  This shows that $\psi_\alpha$ lifts $\varphi_\alpha$ pointwise, furnishing the desired contradiction 
  and completing the proof.  
\end{proof}

\section{Cohen forcing} \label{cohen_sec}
Given an infinite set $X$, we let $\Add(\omega, X)$ denote the forcing to add 
$|X|$-many Cohen reals. Concretely, $\Add(\omega, X)$ consists of all finite partial 
functions from $X$ to $\{0,1\}$, ordered by reverse inclusion.

In this section, we will modify the arguments of the previous section to prove the following theorem:

\begin{theorem} \label{thm:cohen_forcing}
  Suppose that $A$ is a nonfree abelian group and $\kappa = |A|$. Then, in $V^{\Add(\omega, \kappa)}$, 
  $A$ is not Whitehead.
\end{theorem}

Note that this establishes clause (1) of Theorem \ref{main_thm}. Before proving Theorem \ref{thm:cohen_forcing}, we make the following observation.

\begin{lemma} \label{lem:cohen_preservation}
  Suppose that $\P$ is a ccc forcing notion and $A$ is an 
  abelian group. Then $A$ is free in $V$ if and only if $A$ is free in $V^{\P}$.
\end{lemma}

\begin{proof}
  If $A$ is free in $V$, then it clearly remains free in $V^{\P}$, and if $A$ is not torsion-free, then it is 
  nonfree in both models. Therefore, we may assume that $A$ is torsion-free and nonfree in $V$, and 
  we will show that $A$ remains nonfree in $V^{\P}$. 
  
  The proof is by induction on $|A|$. Suppose first that $A$ is countable. Then, by Theorem 
  \ref{thm:pontrjagin}, $A$ has a finitely generated subgroup $X$ whose pure closure is not finitely 
  generated. Since $A$ has the same finitely generated subgroups in $V$ and in $V^{\P}$, and 
  since a subgroup of $A$ is pure in $V$ if and only if it is pure in $V^{\P}$, $X$ retains these 
  properties in $V^{\P}$, so, again by Theorem \ref{thm:pontrjagin}, $A$ is nonfree in $V^{\P}$.
  
  Now suppose that $|A| = \kappa > \aleph_0$. First, note that we may assume that $A$ is almost 
  free. Indeed, if $A_0$ is a nonfree subgroup of $A$ and $|A_0| < \kappa$, then we can 
  apply the inductive hypothesis to conclude that $A_0$, and hence also $A$, remains 
  nonfree in $V^{\P}$.
  
  Thus, assume that $A$ is almost free, and hence, by Shelah's aforementioned singular 
  compactness theorem, $\kappa$ is regular. Since $A$ is nonfree, we can 
  fix a $\subseteq$-increasing, continuous sequence $\vec{A} := \langle A_\alpha \mid \alpha < 
  \kappa \rangle$ 
  such that
  \begin{enumerate}
    \item each $A_\alpha$ is a subgroup of $A$ of cardinality less than $\kappa$;
    \item $A = \bigcup_{\alpha < \kappa} A_\alpha$;
    \item the set $S := \{\alpha < \kappa \mid A_{\alpha+1}/A_\alpha \text{ is nonfree}\}$ is 
    stationary in $\kappa$.
  \end{enumerate}
  Items (1) and (2) are clearly upward absolute to $V^{\P}$. By the inductive hypothesis, for 
  all $\alpha \in S$, the group $A_{\alpha+1}/A_\alpha$ remains nonfree in $V^{\P}$. Moreover, 
  since $\P$ has the ccc, $S$ remains stationary in $\kappa$, so (3) is also upward absolute. 
  Thus, $\vec{A}$ continues to witness that $A$ nonfree in $V^{\P}$.
\end{proof}

Given infinite cardinals $\kappa \leq \lambda$, note that $\Add(\omega, \kappa)$ is a regular 
suborder of $\Add(\omega, \lambda)$, and $\Add(\omega, \lambda) \cong \Add(\omega, \kappa) 
\times \Add(\omega, [\kappa, \lambda))$. To establish Theorem \ref{thm:cohen_forcing}, we will 
actually prove the following formally stronger statement.

\begin{theorem} \label{thm:cohen}
  Suppose that $A$ is a nonfree abelian group, $\kappa$ is the least cardinality 
  of a nonfree subgroup of $A$,
  \[
    0 \ra K \ra F \ra A \ra 0
  \]
  is a free resolution of $A$, and $B_K$ is a basis for $K$. Then, in $V^{\Add(\omega, \kappa)}$, 
  there is a $\varphi \in \Hom(K,\Z)$ such that
  \begin{enumerate}
    \item $\varphi[B_K] \subseteq \{0,1\}$;
    \item if $\Q \in V$ is a ccc forcing notion, then, in 
    $V^{\Add(\omega, \kappa) \times \Q}$, there is no 
    $\psi \in \Hom(F,\Z)$ extending $\varphi$.
  \end{enumerate}
\end{theorem}

\begin{proof}
  The proof is by induction on $|A|$; since the proof resembles that of Theorem \ref{thm:main_thm_1}, 
  we will omit some details and frequently refer the reader back to that earlier proof.
  If $A$ has nonzero torsion elements or is torsion-free and countable, then the homomorphism 
  $\varphi \in \Hom(K,\Z)$ constructed (in $V$) in the proof of Proposition \ref{prop:torsion} 
  or Theorem \ref{ctbl_thm}, respectively, satisfies the conclusion of the theorem, since, as 
  the reader may readily verify in either case, the argument that $\varphi$ cannot be extended 
  to an element of $\Hom(F,\Z)$ can be carried out equally well in any outer model of $V$.
  
  We may thus assume that $\kappa > \aleph_0$ and $A$ is torsion-free. 
  As in the proof of Theorem \ref{thm:main_thm_1}, 
  we may assume that $A$ is almost free and $|F| = \kappa$. To see this, let 
  $N$, $A_0$, $F_0$, $K_0$, $B_{F,0}$, and $B_{K,0}$ be as in the analogous argument 
  in the earlier proof. Suppose that we are able to obtain, 
  in $V^{\Add(\omega, \kappa)}$, a $\varphi_0 \in \Hom(K_0, \Z)$ such that 
  $\varphi_0[B_{K,0}] \subseteq \{0,1\}$ and, for every ccc $\Q \in V$, $\varphi_0$ cannot be extended to a map $\psi_0 \in 
  \Hom(F_0, \Z)$ in $V^{\Add(\omega, \kappa) \times \Q}$. Extend $\varphi_0$ 
  to a map $\varphi \in \Hom(K,\Z)$ by letting $\varphi(t) = 0$ for all $t \in B_K \setminus B_{K,0}$. Then 
  $\varphi \in V^{\Add(\omega, \kappa)}$. Moreover, if 
  $\Q \in V$ is a ccc forcing notion, then $\varphi$ does not 
  extend in $V^{\Add(\omega, \kappa) \times \Q}$ to an element of $\Hom(F,\Z)$, as the 
  restriction of such an extension to $F_0$ would extend $\varphi_0$, contradicting our choice of 
  $\varphi_0$.
  
  Assume now that $A$ is almost free and $|F| = \kappa$. For $\alpha < \kappa$, let 
  $M_\alpha$, $\delta_\alpha$, $A_\alpha$, $F_\alpha$, and $K_\alpha$ be as in the proof of 
  Theorem \ref{thm:main_thm_1}. Let
  \[
    S := \{\alpha < \kappa \mid A_{\alpha+1}/A_\alpha \text{ is nonfree}\};
  \] 
  again, since $A$ is almost free and nonfree, we may assume that $S$ is stationary in $\kappa$.
  For each $\alpha \in S$, let $\kappa_\alpha$, $A_\alpha^*$, $F_\alpha^*$, $K_\alpha^*$, 
  $B^*_{F,\alpha}$, and $B^*_{K,\alpha}$ be as in the proof of Theorem \ref{thm:main_thm_1}.
  
  Note that $\Add(\omega, \kappa) \cong \Add(\omega, \kappa) \times \Add(\omega, \kappa)$; move to an extension $V_1$ of the ground model $V$ by (the first copy 
  of) $\Add(\omega, \kappa)$. In particular, by the inductive hypothesis, we can assume that, for every 
  $\alpha \in S$, we have a $\varphi_\alpha \in \Hom(K_\alpha^*, \Z)$ such that
  \begin{itemize}
    \item $\varphi_\alpha[\{t + K_\alpha \mid t \in B^*_{K,\alpha}\}] \subseteq \{0,1\}$;
    \item for every ccc $\Q \in V$, in $V_1^{\Q}$ there is no $\psi_\alpha \in \Hom(F_\alpha^*, \Z)$ 
    extending $\varphi_\alpha$.
  \end{itemize}
  We now describe a method for producing elements of $\Hom(K,\Z)$. Given a function $f:\kappa \ra \{0,1\}$, 
  define $\varphi(f) \in \Hom(K,\Z)$ as follows. For every $t \in B_K$, if there is $\alpha \in S$ 
  such that $t \in B^*_{K,\alpha}$ and $f(\alpha) = 1$, then let $\varphi(f)(t) = \varphi_\alpha(t 
  + K_\alpha)$. In all other cases, let $\varphi(f)(t) = 0$. Then extend linearly to all of $K$.
  By construction, we have $\varphi(f)[B_K] \subseteq \{0,1\}$ for all $f:\kappa \ra \{0,1\}$.
  
  Note that this construction continues to make sense in any outer model of $V_1$. Let 
  $\dot{g}$ be the canonical $\Add(\omega, \kappa)$-name for the union of the generic filter, and 
  note that $\dot{g}$ is a name for a function from $\kappa$ to $\{0,1\}$. Then let $\dot{\varphi}$ 
  be an $\Add(\omega, \kappa)$-name for $\varphi(\dot{g})$. We claim that $\dot{\varphi}$ is forced 
  to be as desired. To see this, let $\Q \in V$ be a ccc forcing 
  notion, and suppose for the sake of contradiction that $\dot{\psi} \in V_1$ is an $\Add(\omega, \kappa) \times 
  \Q$-name that is forced to be an element of $\Hom(F,\Z)$ extending $\dot{\varphi}$.
  
  Let $G \times H$ be $\Add(\omega, \kappa) \times \Q$-generic over $V_1$, and let 
  $g = \bigcup G$. Let $\varphi = \dot{\varphi}^G = \varphi(g)$, and let $\psi = \dot{\psi}^{G\times H}$.
  Let $\{z_\gamma \mid \gamma < \kappa\}$ be the $\vartriangleleft$-least injective enumeration of 
  $B_F$. For each $\gamma < \kappa$, find $(p_\gamma,q_\gamma) \in G \times H$ deciding the value of 
  $\dot{\psi}(z_\gamma)$. Since the domain of each $p_\gamma$ is a finite subset of $\kappa$, we can 
  find $\alpha \in S$ such that
  \begin{itemize}
    \item $\delta_\alpha = \alpha$;
    \item for all $\gamma < \alpha$, we have $\dom(p_\gamma) \subseteq \alpha$.
  \end{itemize}
  For each $p \in \Add(\omega, \kappa)$, define $\hat{p} \in \Add(\omega, \kappa)$ by letting 
  $\dom(\hat{p}) = \dom(p)$ and, for all $\beta \in \dom(p)$, letting
  \[
    \hat{p}(\beta) = \begin{cases}
      p(\beta) & \text{if } \beta \neq \alpha \\ 
      1 - p(\beta) & \text{if } \beta = \alpha.
    \end{cases}
  \]
  Since the map $p \mapsto \hat{p}$ is an automorphism of $\Add(\omega, \kappa)$, the filter 
  $\hat{G} = \{\hat{p} \mid p \in G\}$ is generic over $V_1$, and the function $\hat{g} = \bigcup 
  \hat{G}$ agrees with $g$ everywhere except $\alpha$, where it takes the opposite value.
  Without loss of generality, assume that $g(\alpha) = 0$ and $\hat{g}(\alpha) = 1$. For ease of 
  comprehension, we then denote $(G,g)$ by $(G_0, g_0)$ and $(\hat{G}, \hat{g})$ by 
  $(G_1, g_1)$. Note also that $V_1[G_0] = V_1[G_1]$, and hence $H$ is $\Q$-generic over both.
  For $i < 2$, let $\varphi_i := \varphi(g_i) = \dot{\varphi}^{G_i}$, and let $\psi_i := 
  \dot{\psi}^{G_i \times H}$. 
  
  By our assumptions about $\dot{\psi}$, we know that, for each $i < 2$, $\psi_i \in \Hom(F,\Z)$ 
  extends $\varphi_i$. Moreover, by our choice of $\alpha$ and the fact that $g_0 \restriction \alpha 
  = g_1 \restriction \alpha$, we know that $\psi_0 \restriction F_\alpha = \psi_1 \restriction 
  F_\alpha$. Now define a homomorphism $\psi_\alpha \in \Hom(F_\alpha^*, \Z)$ as follows. For all 
  $z \in B^*_{F,\alpha}$, let $\psi_\alpha(z + F_\alpha) = \psi_1(z) - \psi_0(z)$, and then extend 
  linearly to the rest of $F_\alpha^*$. We claim that $\psi_\alpha$ extends $\varphi_\alpha$. The proof 
  of this is exactly the same as that of the analogous fact in the proof of Theorem \ref{thm:main_thm_1}, 
  so we leave it to the reader. But then, in $V_1[G \times H]$, we have an element of 
  $\Hom(F_\alpha^*,\Z)$ extending~$\varphi_\alpha$, contradicting our choice of $\varphi_\alpha$ and the fact that $V_1[G \times H]$ is an extension of~$V_1$ 
  by $\Add(\omega, \kappa) \times \Q$, which is a ccc forcing notion in $V$.
\end{proof}

\section{Stone spaces} \label{stone_sec}
In this section, we take a direct route to the proof of clause 
(2) of Theorem \ref{main_thm}. A more indirect route, involving 
a broader investigation into more general connections between 
condensed mathematics and forcing, occupies Sections 
\ref{top_spaces_sec} and \ref{compact_sec}, which can be read 
independently of this one. 

Given an infinite cardinal $\kappa$, let $\bb{B}_\kappa$ denote the Boolean completion of 
$\Add(\omega, \kappa)$, and let $S_\kappa$ denote the Stone space of $\bb{B}_\kappa$. We will sometimes slightly abuse notation and think 
of $\Add(\omega, \kappa)$ as a subset of $\bb{B}_\kappa$. Given sets $X$ and $Y$, we now highlight a 
translation between $\bb{B}_\kappa$-names for functions from $X$ to $Y$ and continuous 
functions from $S_\kappa$ to the product space $Y^X$. Recall that a clopen basis for $S_\kappa$ is 
given by all sets of the form
\[
  N_b := \{\mc{U} \in S_\kappa \mid b \in \mc{U}\}
\]
for $b \in \bb{B}^+_\kappa$.

\begin{definition} \label{def:translation}
  Suppose that $\kappa$ is an infinite cardinal and $X$ and $Y$ are nonempty sets.
  \begin{enumerate}
    \item Suppose that $\varphi : S_\kappa \ra Y^X$ is continuous, where $Y$ is discrete and 
    $Y^X$ is given the product topology. Define a $\bb{B}_\kappa$-name $\dot{f}_\varphi$ for a 
    function from $X$ to $Y$ as follows. Given $x \in X$ and $y \in Y$, let 
    $U_{x,y} := \{h \in Y^X \mid h(x) = y\}$, and set
    \[
      \llbracket \dot{f}_\varphi(x) = y \rrbracket_{\bb{B}_x} := \bigvee \{b \in \bb{B}_\kappa 
      \mid N_b \subseteq \varphi^{-1}[U_{x,y}]\}.
    \]
    \item Suppose now that $Y$ is a \emph{finite} set and $\dot{f}$ is a 
    $\bb{B}_\kappa$-name for a function from $X$ to $Y$. Define a function 
    $\varphi_{\dot{f}} : S_\kappa \ra Y^X$ as follows: for each $\mc{U} \in S_\kappa$ and $x \in X$, let 
    $\varphi_{\dot{f}}(\mc{U})(x)$ be the unique $y \in Y$ such that $\llbracket \dot{f}(x) = y 
    \rrbracket_{\bb{B}_\kappa} \in \mc{U}$. Note that such a $y$ must exist, since 
    $\{ \llbracket \dot{f}(x) = y \rrbracket_{\bb{B}_\kappa} \mid y \in Y \}$ is a finite maximal 
    antichain in $\bb{B}_\kappa$ and therefore must intersect $\mc{U}$.
  \end{enumerate}
\end{definition}

It is routine to verify that $\dot{f}_\varphi$ as defined in Definition \ref{def:translation}(1) is 
indeed a $\bb{B}_\kappa$-name for a function from $X$ to $Y$ and that $\varphi_{\dot{f}}$ as 
defined in Definition \ref{def:translation}(2) is a continuous function from $S_\kappa$ to 
$Y^X$. (These facts will also follow from more general arguments appearing below in 
Section \ref{top_spaces_sec}.) This translation, 
together with Theorem \ref{thm:cohen}, allows us to replace the space $2^\kappa$ in Theorem 
\ref{thm:main_thm_1} with the Stone space $S_\kappa$.

\begin{theorem} \label{thm:stone}
  Suppose that $A$ is a nonfree abelian group and $\kappa = |A|$. Suppose moreover that 
  $0 \ra K \ra F \ra A \ra 0$ is a free resolution of $A$. Then there is a continuous 
  map $\varphi : S_\kappa \ra \Hom(K,\Z)$ that cannot be lifted pointwise to a~continuous map 
  $\psi : S_\kappa \ra \Hom(F,\Z)$.
  
  Moreover, given a basis $B_K$ for $K$, we can require that, for all $U \in S_\kappa$ and all 
  $t \in B_K$, we have $\varphi(U)(t) \in \{0,1\}$.
\end{theorem}

\begin{proof}
  By Theorem \ref{thm:cohen}, we can fix an $\Add(\omega, \kappa)$-name $\dot{\tau}$ for 
  an element of $\Hom(K,\Z)$ for which it is forced to be the case that
  \begin{itemize}
    \item $\dot{\tau}[B_K] \sub \{0,1\}$; and
    \item there is no $\sigma \in \Hom(F,\Z)$ that extends $\dot{\tau}$.
  \end{itemize}
  We think of $\dot{\tau}$ as a $\bb{B}_\kappa$-name. Let $\dot{f}$ be the $\bb{B}_\kappa$-name 
  for the restriction of $\tau$ to $B_K$; $\dot{f}$ is thus a name for a function from 
  $B_K$ to $2$. Let $\varphi_{\dot{f}}$ be the continuous function from $S_\kappa$ to 
  $2^{B_K}$ given by Definition \ref{def:translation}(2), and let $\varphi : S_\kappa \ra 
  \Hom(K,\Z)$ be defined by letting $\varphi(\mc{U})$ be the linear extension of $\varphi_{\dot{f}}(\mc{U})$ 
  to all of $K$. Then $\varphi$ is also a~continuous function.
  
  We claim that $\varphi$ is as desired. We clearly have $\varphi(\mc{U})[B_K] \sub \{0,1\}$ for all 
  $\mc{U} \in S_\kappa$. Suppose for the sake of contradiction that $\psi:S_\kappa \ra \Hom(F,\Z)$  
  is a continuous map and extends $\varphi$ pointwise. Let $\dot{f}_\psi$ be the $\bb{B}_\kappa$-name 
  for a function from $F$ to $\Z$ given by Definition \ref{def:translation}(1). 
  
  \begin{claim}
    $\Vdash_{\bb{B}_\kappa} \dot{f}_\psi \in \Hom(F,\Z)$.
  \end{claim}
  
  \begin{proof}
    Suppose for the sake of contradiction that there are $y,z \in F$ and $b \in \bb{B}_\kappa^+$ 
    such that $b \Vdash_{\bb{B}_\kappa} \dot{f}_\psi(y) + \dot{f}_\psi(z) \neq \dot{f}_\psi(y+z)$.
    By extending $b$, we may assume that there are $k_y$, $k_z$, and $k_{yz}$ such that 
    $b \Vdash_{\bb{B}_\kappa} (\dot{f}_\psi(y), \dot{f}_\psi(z), \dot{f}_\psi(y + z)) = 
    (k_y, k_z, k_{yz})$ and $k_y + k_z \neq k_{yz}$. Let $\mc{U} \in S_\kappa$ be such that $b \in \mc{U}$. 
    Then by the definition of $\dot{f}_\psi$, we have $\psi(\mc{U})(y) = k_y$, $\psi(\mc{U})(z) = k_z$, and 
    $\psi(\mc{U})(y + z) = k_{yz}$, contradicting the fact that $\psi(\mc{U}) \in \Hom(F,\Z)$.
  \end{proof}
  
  \begin{claim}
    $\Vdash_{\bb{B}_\kappa} \dot{f}_\psi \restriction K = \dot{\tau}$.
  \end{claim}
  
  \begin{proof}
    Suppose for the sake of contradiction that there is $t \in B_K$ and $b \in \bb{B}_\kappa^+$ 
    such that $b \Vdash_{\bb{B}_\kappa} \dot{f}_\psi(t) \neq \dot{\tau}(t)$. By extending $b$ if 
    necessary, we may assume without loss of generality that there are $k_\psi, k_\tau \in \Z$ 
    such that $b \Vdash_{\bb{B}_\kappa}(\dot{f}_\psi(t),\dot{\tau}(t)) = (k_\psi,k_\tau)$.
    Let $\mc{U} \in S_\kappa$ be such that $b \in \mc{U}$. Then by the definition of 
    $\varphi_{\dot{f}}$, we have $\varphi_{\dot{f}}(\mc{U})(t) = \varphi(\mc{U})(t) = k_\psi$, and by the 
    definition of $\dot{f}_\psi$, we have $\psi(\mc{U})(t) = k_\tau$, contradicting the fact that 
    $\psi$ extends $\varphi$ pointwise.
  \end{proof}
  But now we have shown that there is forced by $\bb{B}_\kappa$ to be a $\sigma \in \Hom(F,\Z)$ (namely, 
  $\dot{f}_\psi$) that extends $\dot{\tau}$, which is a contradiction, thus completing the proof.
\end{proof}

\section{Topological spaces in forcing extensions}
\label{top_spaces_sec}

In this section and the next, we begin a general investigation 
into connections between condensed mathematics and forcing, 
the results of which will yield clause (2) of Theorem \ref{main_thm} 
as a corollary.

Suppose that $(X,\tau)$ is a compact Hausdorff space and $W$ is an outer model of $V$. 
Then there is a canonical way to interpret $(X,\tau)$ as a compact Hausdorff space 
$(\hat{X}, \hat{\tau})$ in $W$. First, let $\tau_c$ denote the collection of all 
closed subsets of $X$ (in $V$). In $W$, let $\hat{X}$ denote the collection of all 
maximal filters on $\tau_c \setminus \{\emptyset\}$. In other words, $\hat{X}$ is the 
collection of all maximal collections of elements of $\tau_c$ with the finite intersection 
property.

There is a natural map $\pi:X \ra \hat{X}$ defined by letting $\pi(x) := \{D \in \tau_c \mid 
x \in D\}$ for all $x \in X$. We also define $\pi:\tau \ra \power(\hat{X})$ by letting 
$\pi(U) := \{F \in \hat{X} \mid (X \setminus U) \notin F\}$. Then let $\hat{\tau}$ be the 
topology on $\hat{X}$ generated by $\pi[\tau]$.

\begin{remark}
    The space $(\hat{X}, \hat{\tau})$ identified above, together with the associated maps 
    $\pi:X \ra \hat{X}$ and $\pi:\tau \ra \hat{\tau}$ is precisely the \emph{interpretation} 
    of $(X,\tau)$ in $W$, in the sense of \cite{zapletal}. In particular, this implies 
    the following properties, which can also readily be directly verified:
    \begin{itemize}
        \item $(\hat{X},\hat{\tau})$ is compact in $W$;
        \item $\pi[\tau]$ is a basis for $\hat{\tau}$;
        \item for all $x \in X$ and $U \in \tau$, we have $x \in U$ if and only if 
        $\pi(x) \in \pi(U)$.
    \end{itemize}
\end{remark}

Suppose that $\B$ is a complete Boolean algebra and $S = S(\B)$ is its Stone space.
We now show how to identify the space $(\hat{X}, \hat{\tau})$, as computed 
in the forcing extension by $\B$, with an appropriate quotient of the space 
$C(S,X)^V$, equipped with the compact-open topology. In fact, we will prove 
something more general.

We first establish a correspondence between elements of $C(S,X)$ and 
$\B$-names for points in $\hat{X}$.\footnote{Recent work in 
this direction, focusing especially on the case in which 
$X$ is a (compact) Polish space, can be found in 
\cite{vaccaro_viale}; many of the basic results we present here 
are implicit in that work.} 
In the forward direction, fix a 
function $f \in C(S,X)$, and define a $\B$-name $\dot{F}_f$ for a subset 
of $\tau_c$ as follows. For every $D \in \tau_c$ and every $b \in \B^+$, 
put $(\check{D}, b)$ in $\dot{F}_f$ if and only if $N_b \subseteq 
f^{-1}[D]$.

\begin{lemma}
    $\dot{F}_f$ is a $\B$-name for a maximal filter on $\tau_c$.
\end{lemma}

\begin{proof}
    The fact that $\dot{F}_f$ is forced to be a filter is immediate from the 
    definition. To show that it is forced to be maximal, fix an arbitrary 
    $C \in \tau_c$. It suffices to show that the following set is dense in 
    $\B^+$:
    \[
        E_C := \{b \in \B^+ \mid N_b \subseteq f^{-1}[C] \text{ \textbf{or} } \exists 
        D \in \tau_c\, [D \cap C = \emptyset \wedge N_b \subseteq f^{-1}[D]]\}.
    \]
    To this end, fix an arbitrary $b_0 \in \B^+$, and assume that $N_{b_0} 
    \not\subseteq f^{-1}[C]$. Fix $\mc U \in N_{b_0}$ such that 
    $f(\mc U) \in X \setminus C$. Since $X$ is compact, it is regular, so 
    we can find an open set $O \in \tau$ such that $\mc U \in O \subseteq 
    \mathrm{cl}(O) \subseteq X \setminus C$. By the continuity of $f$, we can 
    find $b \leq b_0$ such that $\mc U \in N_b$ and $N_b \subseteq 
    f^{-1}[O]$. Then $b \in E_C$, as desired.
\end{proof}

In the other direction, fix a $\B$-name $\dot{F}$ for a maximal filter on 
$\tau_c$. For each $D \in \tau_c$, let $b_D := \llbracket D \in \dot{F} 
\rrbracket_\B$. Given $\mc U \in S$, let $\dot{F}_{\mc U} := \{D \in \tau_c 
\mid b_D \in \mc U\}$.

\begin{lemma}
    For all $\mc U \in S$, $\bigcap \dot{F}_{\mc U}$ contains exactly one point.
\end{lemma}

\begin{proof}
    Fix $\mc U \in S$. We first show that $\dot{F}_{\mc U}$ is a filter. To
    this end, fix $D_0, D_1 \in \dot{F}_{\mc U}$. Since $\dot{F}$ is forced to be a filter, 
    we have $\llbracket D_0 \cap D_1 \in \dot{F} \rrbracket_{\B} = 
    \llbracket D_0 \in \dot{F} \rrbracket_{\B} \wedge \llbracket D_1 \in \dot{F}
    \rrbracket_{\B}$. Then, since $\mc U$ is a filter and $\llbracket D_i 
    \in \dot{F} \rrbracket_{\B} \in \mc U$ for $i < 2$, it follows that 
    $\llbracket D_0 \cap D_1 \in \dot{F} \rrbracket_{\B} \in \mc U$, 
    so $D_0 \cap D_1 \in \dot{F}_{\mc U}$.

    Thus, $\dot{F}_{\mc U}$ has the finite intersection property; since $X$ is 
    compact, it follows that $\bigcap \dot{F}_{\mc U}$ is nonempty. Now suppose 
    for the sake of contradiction that $x_0$ and $x_1$ are distinct elements of 
    $\bigcap \dot{F}_{\mc U}$. Since $X$ is Hausdorff, we can find disjoint open 
    neighborhoods $U_0$ and $U_1$ of $x_0$ and $x_1$, respectively. 
    For $i < 2$, let $D_i := X \setminus U_i$. Since $D_0 \cup D_1 = X$ and 
    $\dot{F}$ is forced to be a maximal filter on $\tau_c \setminus \{\emptyset\}$, 
    a routine argument shows that $\llbracket D_0 \in \dot{F} \rrbracket_{\B} 
    \vee \llbracket D_1 \in \dot{F} \rrbracket_{\B} = 1_{\B}$. Therefore, 
    since $\mc U$ is an ultrafilter, there must be $i < 2$ such that 
    $\llbracket D_i \in \dot{F} \rrbracket_{\B} \in \mc U$, and hence 
    $D_i \in \dot{F}_{\mc U}$. But then $x_i \notin D_i$ and 
    $\bigcap \dot{F}_{\mc U} \subseteq D_i$, so $x_i \notin 
    \bigcap \dot{F}_{\mc U}$, which is a contradiction. It follows that 
    $\bigcap \dot{F}_{\mc U}$ contains precisely one point.
\end{proof}

Using the claim, define a function $f_{\dot{F}}:S \ra X$ by letting, for all 
$\mc U \in S$, $f_{\dot{F}}(\mc U)$ be the unique $x \in \bigcap \dot{F}_{\mc U}$.

\begin{lemma}
    $f_{\dot{F}}$ is continuous.
\end{lemma}

\begin{proof}
    Fix $U \in \tau$ and $\mc U \in f_{\dot{F}}^{-1}[U]$.
    For every $y \in X \setminus U$, we can find $D_y \in \tau_c$ 
    such that $y \notin D_y$ and $b_{D_y} \in \mc U$. Then 
    $\{X \setminus D_y \mid y \in X \setminus U\}$ is an open cover 
    of $X \setminus U$, so we can find a finite subset 
    $a \subseteq X \setminus U$ such that $D := \bigcap\{D_y \mid y \in 
    a\} \subseteq U$. Then $b_D = \bigwedge \{b_{D_y} \mid y \in a\} 
    \in \mc U$ and $N_{b_D} \subseteq f_{\dot{F}}^{-1}[U]$.
\end{proof}

It is routine to verify that this is a reflexive duality: Given 
$g \in C(S,X)$, we have $f_{\dot{F}_g} = g$ and, given a $\B$-name 
$\dot{E}$ for a maximal filter on $\tau_c$, we have 
$\Vdash_{\B} \dot{E} = \dot{F}_{f_{\dot{E}}}$. Also, this story can 
clearly be relativized below a given condition $b \in \B^+$: functions 
in $C(N_b,X)$ correspond to $\B$-names $\dot{F}$ such that 
$b \Vdash_{\B} ``\dot{F} \text{ is a maximal filter on } \tau_c"$.

Now suppose that $X$ is an arbitrary Hausdorff space, not necessarily 
compact. Now there may or may not be a well-behaved interpretation 
$\hat{X}$ of $X$ in the extension by $\B$. However, there is still a natural
extension of $X$ in $V^{\B}$ that will correspond to a quotient of 
the space $C(S,X)^V$ in $V^{\B}$. We will denote this extension by 
$\hat{X}_{\mc K}$; it can be thought of as the extension of $X$ spanned by 
ground model compact sets. We now give a definition of this space.

Let $\mc K(X)$ denote the collection of compact subsets of $X$ 
(in $V$). Then $(\mc K(X), \subseteq)$ is a directed partial order. 
For each $K \in \mc K(X)$, let $\tau_K$ denote the subspace topology 
on $K$. For each $K \in \mc K(X)$, there is a canonical interpretation 
$(\hat{K}, \hat{\tau}_K)$ of $(K, \tau_K)$ in $V^{\B}$, as above. 
Namely, $\hat{K}$ is the collection of all maximal filters on $(\tau_K)_c$.
Moreover, if $K \subseteq K'$ are both in $\mc K(X)$, then the inclusion map 
$\iota_{KK'}: K \ra K'$ induces an inclusion map $\hat{\iota}_{KK'} : 
\hat{K} \ra \hat{K'}$. We now define $(X_{\mc K}, \tau_{\mc K})$ in 
$V^{\B}$. 

First, we let the underlying set, $X_{\mc K}$, be the direct 
limit (i.e., colimit) of the system $\langle \hat{K}, \hat{\iota}_{KK'} 
\mid K,K' \in \mc K(X), \ K \subseteq K' \rangle$. Formally, 
$X_{\mc K}$ consists of equivalence classes of the form 
$[(x,K)]$, where $K \in \mc K(X)$ and $x \in \hat{K}$. Given an open 
set $U \in \tau$ and $K \in \mc K(X)$, let $\hat{U}_K$ denote its 
interpretation in $\hat{\tau}_K$. We then define $\hat{U}_{\mc K} 
\subseteq \hat{X}_{\mc K}$ as follows: given $K \in \mc K(X)$ and 
$x \in \hat{K}$, put $[(x,K)] \in \hat{U}_{\mc K}$ if and only if 
$x \in \hat{U}_K$. It is routine to check that this is well-defined. 
Now let $\hat{\tau}_{\mc K}$ be the topology generated by 
$\{\hat{U}_{\mc K} \mid U \in \tau\}$.

We take a slight detour to note that this is not the only, nor arguably the 
most natural, topology we could have placed on $\hat{X}_{\mc K}$. 
We alternatively could have taken the direct limit topology, which we will 
denote by $\hat{\tau}^{\lim}_{\mc K}$, i.e., the finest topology that makes 
all of the limit maps $\hat{\iota}_K : \hat{K} \ra \hat{X}_{\mc K}$ continuous.
We always have $\hat{\tau}_{\mc K} \subseteq \hat{\tau}^{\lim}_{\mc K}$, and 
if $X$ is locally compact, then the two topologies coincide (moreover, 
if $X$ is locally compact, then $\hat{X}_{\mc K}$ is precisely the interpretation 
of $X$ in $V^{\B}$ in the sense of \cite{zapletal}). However, in general the 
direct limit topology can be strictly finer, even for such relatively nice 
spaces as ${^\omega}\omega$, as the following example shows.

\begin{proposition}
    Suppose that $X = {^\omega}\omega$ and $\B$ adds a dominating real. 
    Then $\hat{\tau}^{\lim}_{\mc K} \supsetneq \hat{\tau}_{\mc K}$ in 
    $V^{\B}$.
\end{proposition}

\begin{proof}
    Let $W$ denote the extension by $\B$. 
    In $V$, for every $x \in {^\omega}\omega$ let $K_x := \{y \in {^{\omega}\omega} 
    \mid y < x\}$. Then $\{K_x \mid x \in {^\omega}\omega\}$ is a $\subseteq$-cofinal 
    subset of $\mc K(X)$. Moreover, it is routine to show that, for every 
    $x \in {^\omega}\omega$, we have $\hat{K}_x = (\prod_{n < \omega} x(n))^W$.
    Therefore, the underlying set of $\hat{X}_{\mc K}$ can be identified with 
    $\{y \in ({^\omega}\omega)^W \mid \exists x \in ({^\omega}\omega)^V ~ y < x\}$.

    Let $d \in ({^\omega}\omega)^W$ dominate every real in $V$, and let 
    \[
      O = \{y \in \hat{X}_{\mc K} \mid \forall n < \omega ~ y(n) \neq d(n)\}.
    \]

    \begin{claim}
        $O \in \hat{\tau}^{\lim}_{\mc K}$.
    \end{claim}

    \begin{proof}
        It suffices to show that, for every $x \in ({^\omega}\omega)^V$, 
        $O \cap \hat{K}_x \in \hat{\tau}_{K_x}$. Fix such an $x$, and let 
        $A := \{n < \omega \mid d(n) < x(n)\}$. By the choice of $d$, $A$ is 
        finite. Then, in $V$, the set $U = \{y \in {^\omega}\omega \mid 
        \forall n \in A ~ y(n) \neq d(n)\}$ is in $\tau$, and we have 
        $O \cap \hat{K}_x = \hat{U}_{K_x}$.
    \end{proof}

    \begin{claim}
        There is no $U \in \tau$ such that $\hat{U}_{\mc K} \subseteq O$.
    \end{claim}

    \begin{proof}
        Fix $U \in \tau$. By shrinking $U$ if necessary, we may assume that 
        $U$ is a basic open set, i.e., there is a finite set $A \subseteq \omega$ 
        and a function $\sigma:A \ra \omega$ such that $U = \{y \in {^\omega}\omega 
        \mid y \restriction A = \sigma\}$. Choose $n^* \in \omega \setminus A$ 
        and define $y \in {^\omega}\omega$ by 
        \[
          y(n) = \begin{cases}
              \sigma(n) & \text{if } n \in A \\
              d(n^*) & \text{if } n = n^* \\
              0 & \text{otherwise.}
          \end{cases}
        \]
        Then, in $W$, we have $y \in \hat{U}_{\mc K} \setminus O$, so $\hat{U}_{\mc 
        K} \not\subseteq O$.
    \end{proof}
    Since $\{\hat{U}_{\mc K} \mid U \in \tau\}$ is a base for $\hat{\tau}_{\mc K}$, 
    this completes the proof.
\end{proof}

We now show that $\hat{X}_{\mc K}$ is homeomorphic to a natural quotient of 
$C(S,X)^V$ in the extension by $\B$. Recall that we endow $C(S,X)^V$ with the 
compact-open topology. Now suppose that $G$ is a 
generic ultrafilter on $\B$, and, in $V[G]$, define an equivalence relation 
$\sim_G$ on $C(S,X)^V$ by letting $g \sim_G h$ if and only if there is 
$b \in G$ such that $g \restriction N_b = h \restriction N_b$. Let 
$C(S,X)^V/G$ denote the quotient space with respect to this equivalence 
relation. For each $U \in \tau$, define a set $N_U \subseteq C(S,X)^V/G$ 
as follows: for all $g \in C(S,X)^V$, put $[g] \in N_U$ if and only if there 
is $b \in G$ such that $g[N_b] \subseteq U$; note that this is well-defined.

\begin{proposition}
    $\{N_U \mid U \in \tau\}$ is a base for $C(S,X)^V/G$.
\end{proposition}

\begin{proof}
    It is immediate from the definition that each $N_U$ is open. Now 
    let $\hat{O} \subseteq C(S,X)^V/G$ be open, and fix $g \in 
    C(S,X)^V$ such that $[g] \in \hat{O}$. Let $O = \{h \in C(S,X)^V \mid 
    [h] \in \hat{O}\}$. We can then find a compact $K \subseteq S$ and 
    an open set $U \in \tau$ such that $g[K] \subseteq U$ and 
    $\{h \in C(S,X)^V \mid h[K] \subseteq U\} \subseteq O$. 
    Since $K$ is compact and $g$ is continuous, we can find a finite 
    collection $b_0, \ldots, b_{n-1}$ from $\B^+$ such that 
    $K \subseteq \bigcup\{N_{b_i} \mid i < n\}$ and, for all $i < n$, 
    we have $g[N_{b_i}] \subseteq U$. Letting $b := \bigvee \{b_i \mid 
    i < n\}$, we get $K \subseteq N_b$ and $g[N_b] \subseteq U$.

    Suppose first that $b \in G$. In this case, we claim that $N_U 
    \subseteq \hat{O}$. To see this, fix $h \in C(S,X)^V$ such that 
    $[h] \in N_U$. We can therefore find $c \in G$ such that 
    $h[N_c] \subseteq U$. Define $h' \in C(S,X)^V$ by letting 
    $h' \restriction N_c = h \restriction N_c$ and letting 
    $h' \restriction (S \setminus N_c)$ be constant, taking an 
    arbitrary value in $U$. Then $h'[K] \subseteq U$, so $[h'] 
    \in \hat{O}$. Moreover, since $c \in G$, we have $h' \sim_G h$, 
    so $[h'] = [h]$. Thus, $N_U \subseteq \hat{O}$. 

    Suppose next that $b \notin G$. In this case, we claim that 
    $\hat{O}$ is actually all of $C(S,X)^V/G$. To see this, 
    let $h \in C(S,X)^V$ be arbitrary. Define $h' \in C(S,X)^V$ 
    by letting $h' \restriction N_b = g \restriction N_b$ and 
    $h' \restriction (S \setminus N_b) = h \restriction (S \setminus N_b)$. 
    Then $h'[N_b] = g[N_b] \subseteq U$, so $[h'] \in \hat{O}$. 
    Since $b \notin G$, we have $[h'] = [h]$, so $[h] \in \hat{O}$. 
\end{proof}

Now, in $V[G]$, define a map $k : C(S,X)^V/G \ra \hat{X}_{\mc K}$ as 
follows. Given $g \in C(S,X)^V$, let $K := g[S]$. Then $K \in 
\mc K(X)$, so $g$ corresponds in $V$ to a $\B$-name $\dot{F}_g$ for 
an element of $\hat{K}$. Let $F_g$ be the evaluation of this name 
in $V[G]$, and set $k([g]) = [(F_g, K)]$.

\begin{proposition}
    $k$ is a bijection.
\end{proposition}

\begin{proof}
    To show that $k$ is injective, fix two functions $g,h \in 
    C(S,X)^V$. Familiar arguments yield the fact that, in $V$, the set
    \[
      \{b \in \B^+ \mid g \restriction N_b = h \restriction N_b \text{ 
      \textbf{or} } g[N_b] \cap h[N_b] = \emptyset\}
    \]
    is dense in $\B^+$. The genericity of $G$ then yields the injectivity 
    of $k$. To show that $k$ is surjective, fix an element 
    $[(F,K)] \in \hat{X}_{\mc K}$, and let $\dot{F}$ be a $\B$-name 
    for $F$ that is forced by $1_{\B}$ to be a maximal filter on 
    $(\tau_K)_c$. Then $f_{\dot{F}}$, as defined at the beginning of 
    this section, is in $C(S,X)^V$, and it is routine to verify that 
    $k([f_{\dot{F}}]) = [(F,K)]$.
\end{proof}

The continuity of $k$ and $k^{-1}$ will now follow immediately from the 
following proposition.

\begin{proposition}
    For all $U \in \tau$, we have $k[N_U] = \hat{U}_{\mc{K}}$.
\end{proposition}

\begin{proof}
    Suppose that $g \in C(S,X)^V$ and $[g] \in N_U$. Let 
    $K := g[S]$, and fix $b \in G$ 
    such that $g[N_b] \subseteq U$. Then, by construction, we have 
    $(K \setminus U) \notin F_g$, so $k([g]) = [(F_g, K)] \in 
    \hat{U}_{\mc K}$. Conversely, suppose that $g \in C(S,X)^V$ and 
    $[g] \notin N_U$. Again, let $K := g[S]$. By familiar arguments, 
    the following set is dense in $\B^+$ (in $V$):
    \[
      \{b \in \B^+ \mid g[N_b] \subseteq U \text{ \textbf{or} }
      g[N_b] \subseteq K \setminus U\}.
    \]
    Since $[g] \notin N_U$, the genericity of $G$ implies that 
    $K \setminus U \in F_g$, so $k([g]) = [(F_g, K))] 
    \notin \hat{U}_{\mc K}$.
\end{proof}

Altogether, we have established the following theorem.

\begin{theorem} \label{homeo_thm}
    Suppose that $X$ is a Hausdorff space, $\B$ is a complete Boolean algebra, and $G \subseteq \B$ is a generic 
    ultrafilter. Then, in $V[G]$, we have
    \[
      C(S(\B),X)^V/G \cong \hat{X}_{\mc K}.
    \]
\end{theorem}

We end this section with a few remarks. First, as was mentioned above, 
if $X$ is a locally compact Hausdorff space, then $\hat{X}_{\mc{K}}$ 
is the canonical interpretation of $X$ in $V^{\B}$ in the sense of 
\cite{zapletal}. Recall also that the 
category $\mathsf{ED}$ is precisely the category 
of Stone spaces of complete Boolean algebras, and that $\ul{X}(S)$ equals, by definition, $C(S,X)$ for any $S \in \mathsf{ED}$. Therefore, the results of this section 
show that one way to think about the embedding of the category of 
locally compact Hausdorff spaces into $\Cond$ is to recognize that, as noted, for all such spaces $X$, the condensed set $\ul{X}$ is simply an organized presentation of all forcing names for points in 
the canonical interpretations of $X$ in all possible set forcing 
extensions.

Second, write $cg$ for the left adjoint of the inclusion of the category of compactly generated topological spaces into the category of topological spaces; recall also that the embedding $X\mapsto\ul{X}$ of the category of $T_1$ topological spaces into \textsf{Cond} possesses a left adjoint whose counit is given by $X\mapsto cg(X)$ (modulo a small technical point; see \cite[pp.\ 8-9 and Appendix to Lecture II]{CS1}).
Observe next that in the $\bb{B}=\{0,1\}$ (and hence $S(\bb{B})=*$) instance of Theorem \ref{homeo_thm}, $\hat{X}_{\mc K}$ is homeomorphic to $X$.
Had we endowed it instead with the $\hat{\tau}^{\lim}_{\mc K}$ topology considered above, it would be homeomorphic to $cg(X)$, and the $\hat{\tau}^{\lim}_{\mc K}$ topology corresponds more generally to that given by post-composing the conversion recorded in Theorem \ref{homeo_thm} with the functor $cg$.

Third, we note that the 
operation $X \mapsto \hat{X}_{\mc K}$ commutes with products. Namely, 
if $I$ is an index set and $X_i$ is a Hausdorff space for every $i \in I$, then 
$(\widehat{\prod_{i \in I} X_i})_{\mc K} \cong \prod_{i \in I} (\hat{X}_i)_{\mc K}$.
This is because the collection 
\[
  \{ \prod_{i \in I} K_i \mid \forall i \in I \ K_i \in \mc K(X_i)\}
\]
is $\subseteq$-cofinal in $\mc K(\prod_{i \in I} X_i)$.

Fourth, even though Theorem \ref{homeo_thm} holds for \emph{all} 
Hausdorff spaces $X$, there is an important sense in which the correspondence is 
meaningfully stronger in case $X$ is compact. In particular, let 
$\dot{k}$ be the $\B$-name for the homeomorphism we constructed to 
witness Theorem \ref{homeo_thm}. If $X$ is compact, 
then, for every $\B$-name $\dot{x}$ for an element of $\hat{X}_{\mc K} (= 
\hat{X})$, there is a $g \in C(S,X)$ such that $1_{\B} \Vdash \dot{k}([g]) = 
\dot{x}$. In particular, in this case we can take $\dot{x}$ to be a 
$\B$-name for a maximal filter on $\tau_c$, then let 
$g = f_{\dot{x}}$, as defined at the outset of this section.
However, in general this may not be the case, and there may be 
$\B$-names $\dot{x}$ for elements of $\hat{X}_{\mc K}$ for which there is 
\emph{no} $g \in C(S,X)$ such that $1_{\B} \Vdash \dot{k}([g]) = \dot{x}$. 
For such $\dot{x}$, one may need to work below some condition $b \in \B$ that 
forces $\dot{x}$ to lie inside the interpretation of some particular 
$K \in \mc K(X)$; once one does so, one may then find $g \in C(S,X)$ such 
that $b \Vdash \dot{k}([g]) = \dot{x}$. This issue will lead to some 
slight complications and asymmetries in the next section.

We end this section by noting that these observations are of wider significance than might initially be suspected, by reason of the variety of expressions within condensed mathematics in which sets of the form $C(S,X)$, sometimes covertly, partake.
The reader is referred to Lemma \ref{ext_calc_lemma} below, and its argument, for a representative instance.

\section{Compact interpretations and continuous functions}
\label{compact_sec}

Fix for now a complete Boolean algebra $\B$, and let $S = S(\B)$.
In this section, we establish an equivalence between the settings of 
$C(S,X)$ on the one hand and $\B$-names for elements of $\hat{X}_{\mc K}$ 
on the other, with respect to interpretations of certain sentences. 
Fix a positive integer $n$ and Hausdorff spaces $\langle X_i 
\mid i < n \rangle = \vec{X}$. We will specify a collection 
$\mc L(\vec{X})$ of formulas. We will do this by recursively defining 
$\mc L_j(\vec{X})$ for $j \leq n$ and then letting 
$\mc L(\vec{X}) := \bigcup_{j \leq n} \mc L_j(\vec{X})$.
For each $j \leq n$ and $\varphi \in \mc L_j(\vec{X})$, the 
free variables in $\varphi$ will be precisely $\{v_i \mid j \leq i < n\}$.

First, let $\mc L_0(\vec{X})$ consist of all formulas of the form 
$(v_0, v_1, \ldots, v_{n-1}) \in U$, where $U$ is an open subset of 
$\prod_{i < n} X_i$. Now suppose that $j < n$ and we have specified 
$\mc L_j(\vec{X})$. Then $\mc L_{j+1}(\vec{X})$ consists precisely of 
the formulas that can be formed in one of the following ways. First, 
fix a formula $\varphi(v_j, v_{j+1}, \ldots, v_{n-1}) \in \mc L_j(\vec{X})$. 
Then
\begin{enumerate}
    \item the formula $\forall x \varphi(x, v_{j+1}, \ldots, v_{n-1})$ 
    is in $\mc L_{j+1}(\vec{X})$;
    \item the formula $\exists x \varphi(x, v_{j+1}, \ldots, v_{n-1})$ 
    is in $\mc L_{j+1}(\vec{X})$;
    \item for every compact $K \subseteq X_j$, the formula 
    $\exists x \in K \varphi(x, v_{j+1}, \ldots, v_{n-1})$ is in 
    $\mc L_{j+1}(\vec{X})$.
\end{enumerate}
We call a formula $\varphi \in \mc L(\vec{X})$ \emph{bounded} if all of 
its existential quantifiers are bounded by compact sets (i.e., they arise 
through option (3) above rather than option (2)). 
In their most basic setting, that of 
$\prod_{i < n} X_i$, the free variables $v_i$ occurring in formulas in 
$\mc L(\vec{X})$ should be understood as standing for points in $X_i$, 
and the truth value of sentences arising from the substitution of these points 
for the free variables is evaluated in the natural way. But we also want to 
interpret these formulas in other contexts, namely in the context of 
$C(S, \prod_{i < n} X_i)$ and of $V^{\B}$. 

Let us first deal with $C(S, \prod_{i<n} X_i)$, where the free variables $v_i$ 
stand for elements of $C(S, X_i)$. We describe now what it means for 
$C(S, \prod_{i<n} X_i)$ to satisfy a sentence formed by substituting 
appropriate continuous functions for free variables in formulas in 
$\mc L(\vec{X})$. By induction on $j \leq n$, we will deal with 
$\mc L_j(\vec{X})$, in fact specifying what it means for a sentence to 
be satisfied \emph{below a condition $b \in \B$}.

Suppose that $\varphi \in \mc L_0(\vec{X})$, $b \in \B$, and, for all $i < n$, 
$f_i \in C(S, X_i)$. Then $\varphi$ is of the form $(v_0, \ldots, v_{n-1}) 
\in U$ for some open set $U \subseteq \prod_{i < n} X_i$. We say that 
$C(S, \prod_{i < n}X_i)$ \emph{satisfies $\varphi(f_0, \ldots, f_{n-1})$ below 
$b$}, written 
\[
  C(S, \prod_{i < n} X_i) \models_b \varphi(f_0, \ldots, f_{n-1}),
\]
if and only if the set $\{s \in N_b \mid (f_0(s), \ldots, f_{n-1}(s)) \in U\}$ is 
dense in $N_b$. If $b = 1_{\B}$, then it is omitted from the terminology and notation.

Now suppose that $j < n$, $\varphi \in \mc L_{j+1}$, $b \in \B$, and, for all 
$j < i < n$, $f_i \in C(S, X_i)$. 
\begin{itemize}
    \item If $\varphi$ is of the form $\forall x \psi(x, v_{j+1}, \ldots, v_{n-1})$, 
    then 
    \[ C(S, \prod_{i < n} X_i) \models_b \varphi(f_{j+1}, \ldots, f_{n-1}) \] if 
    and only if, for all $g \in C(N_b, X_j)$, we have \[ C(S, \prod_{i < n} X_i) 
    \models_b \psi(g, f_{j+1}, \ldots, f_{n-1}).\]
    \item If $\varphi$ is of the form $\exists x \psi(x, v_{j+1}, \ldots, v_{n-1})$, 
    then 
    \[ C(S, \prod_{i < n} X_i) \models_b \varphi(f_{j+1}, \ldots, f_{n-1}) \] if 
    and only if there exists $g \in C(N_b, X_j)$ such that \[ C(S, \prod_{i < n} X_i) 
    \models_b \psi(g, f_{j+1}, \ldots, f_{n-1}).\]
    \item If $K \subseteq X_j$ is compact and $\varphi$ is of the form $\exists x \in K 
    \psi(x, v_{j+1}, \ldots, v_{n-1})$, then 
    \[ C(S, \prod_{i < n} X_i) \models_b \varphi(f_{j+1}, \ldots, f_{n-1}) \] if 
    and only if there exists $g \in C(N_b, K)$ such that \[ C(S, \prod_{i < n} X_i) 
    \models_b \psi(g, f_{j+1}, \ldots, f_{n-1}).\]
\end{itemize}

In $V^{\B}$, the free variables $v_i$ stand for $\B$-names for elements of 
$(\hat{X}_i)_{\mc K}$. We now define satisfaction in this context. Suppose 
that $\varphi \in \mc L_0(\vec{X})$ is of the form $(v_0, \ldots, v_{n-1}) 
\in U$, $b \in \B$, and, for all $i < n$, $\dot{x}_i$ is a $\B$-name for 
an element of $(\hat{X}_i)_{\mc K}$. Then $V^{\B} \models_b 
\varphi(\dot{x}_0, \ldots, \dot{x}_{n-1})$ if and only if 
$b \Vdash_{\B} (\dot{x}_0, \ldots, \dot{x}_{n-1}) \in \hat{U}_{\mc K}$.

Now suppose that $j < n$, $\varphi \in \mc L_{j+1}$, $b \in \B$, and, for all 
$j < i < n$, $\dot{x}_i$ is a $\B$-name for an element of $(\hat{X}_i)_{\mc K}$.
\begin{itemize}
    \item If $\varphi$ is of the form $\forall x \psi(x, v_{j+1}, \ldots, v_{n-1})$, 
    then $V^{\B} \models_b \varphi(\dot{x}_{j+1}, \ldots, \dot{x}_{n-1})$ if and 
    only if, for every $\B$-name $\dot{y}$ for an element of $(\hat{X}_j)_{\mc K}$, 
    we have $V^{\B} \models_b \psi(\dot{y}, \dot{x}_{j+1}, \ldots, \dot{x}_{n-1})$.
    \item If $\varphi$ is of the form $\exists x \psi(x, v_{j+1}, \ldots, v_{n-1})$, 
    then $V^{\B} \models_b \varphi(\dot{x}_{j+1}, \ldots, \dot{x}_{n-1})$ if and 
    only if there exists a $\B$-name $\dot{y}$ for an element of 
    $(\hat{X}_j)_{\mc K}$
    such that $V^{\B} \models_b \psi(\dot{y}, \dot{x}_{j+1}, \ldots, \dot{x}_{n-1})$.
    \item If $K \subseteq X_j$ is compact and $\varphi$ is of the form $\exists x \in K 
    \psi(x, v_{j+1}, \ldots, v_{n-1})$, then $V^{\B} \models_b \varphi(\dot{x}_{j+1}, 
    \ldots, \dot{x}_{n-1})$ if and only if there exists a $\B$-name $\dot{y}$ such that 
    $b \Vdash_{\B} \dot{y} \in \hat{K}$ and $V^{\B} \models_b \psi(\dot{y}, 
    \dot{x}_{j+1}, \ldots, \dot{x}_{n-1})$.
\end{itemize}

Note that there is some monotonicity in these definitions. Namely, if 
$c \leq_{\B} b$ and $C(S,\prod_{i < n} X_i) \models_b \varphi$, then $C(S,\prod_{i < n} X_i) \models_c 
\varphi$, and similarly for $V^{\B}$.

Given $i < n$ and $f \in C(S,X_i)$, let $\dot{x}_f$ be a $\B$-name for 
$\dot{k}([f])$, where $\dot{k}$ is a name for the homeomorphism constructed 
to witness Theorem \ref{homeo_thm}. Concretely, if $K = f[S]$, then 
$\dot{x}_f$ is a $\B$-name for $[(\dot{F}_f, K)] \in \hat{X}_{\mc K}$.

\begin{theorem} \label{equivalence_thm}
    Suppose that $j \leq n$, $\varphi \in \mc L_j(\vec{X})$, $b \in \B$ and, 
    for all $j \leq i < n$, we have $f_i \in C(S,X_i)$. 
    \begin{enumerate}
        \item If $C(S, \prod_{i < n} X_i) \models_b \varphi(f_j, \ldots, 
        f_{n-1})$, then $V^{\B} \models_b \varphi(\dot{x}_{f_j}, \ldots, 
        \dot{x}_{f_{n-1}})$.
        \item If $\varphi$ is bounded and $V^{\B} \models_b \varphi(\dot{x}_{f_j}, \ldots, 
        \dot{x}_{f_{n-1}})$, then $C(S, \prod_{i < n} X_i) \models_b \varphi(f_j, \ldots, 
        f_{n-1})$.
    \end{enumerate}
\end{theorem}

\begin{proof}
    The proof is by induction on $j$. Suppose first that $j = 0$. For notational 
    convenience, we can assume here that $n = 1$ by considering the single space 
    $X = \prod_{i < n} X_i$ (recalling that $\widehat{(\prod_{i < n} X_i)}_{\mc K} = \prod_{i < n} (\hat{X}_i)_{\mc K}$). Then there is an open set $U \subseteq X$ such that $\varphi$ is the 
    formula $v_0 \in U$, and we are given a function $f \in C(S,X)$. Let 
    $K := \mathrm{im}(f)$, so $K$ is a compact subset of $X$, and $\dot{x}_f$ is forced to 
    be of the form $[(\dot{F}_f, K)]$.

    To verify (1), suppose that $C(S,X) \models_b \varphi(f)$, and hence 
    $\{s \in N_b \mid f(s) \in U\}$ is dense in $N_b$. Suppose for the sake 
    of contradiction that $b \not\Vdash_{\B} \dot{x}_f \in \hat{U}$. Then, by the definition 
    of $\dot{F}_f$, there is $c \leq b$ such that $(\check{K} \setminus \check{U}, c) \in 
    \dot{F}_f$, and therefore, again by the definition of $\dot{F}_f$, we have 
    $N_c \subseteq f^{-1}[K \setminus U]$, which is a contradiction.

    For (2), suppose that $b \Vdash_{\B} \dot{x}_f \in \hat{U}$, i.e., 
    $b \Vdash_{\B} (K \setminus U) \notin \dot{F}_f$. Suppose for the sake of 
    contradiction that $\{s \in N_b \mid f(s) \in U\}$ is not dense in $N_b$. 
    Then there is $c \leq b$ such that $N_c \subseteq f^{-1}[K \setminus U]$. 
    But then $c \Vdash_{\B} (K \setminus U) \in \dot{F}_f$, which is a contradiction.

    Now suppose that $j = j' + 1$ for some $j' < n$. To verify (1),
    suppose that $C(S,X) \models_b \varphi(f_{j}, \ldots f_{n-1})$ but, for the 
    sake of contradiction, assume that $V^{\B} \not\models_b \varphi(\dot{x}_{f_j}, 
    \ldots, \dot{x}_{f_{n-1}})$. Suppose first 
    that $\varphi$ is of the form $\forall x \psi$ for some $\psi \in \mc L_{j'}(\vec{X})$. 
    By assumption, there is a $\B$-name $\dot{y}$ for an element of $(\hat{X}_{j'})_{\mc K}$ 
    such that $V^\B \not\models_b \psi(\dot{y}, \dot{x}_{f_j}, \ldots, 
    \dot{x}_{f_{n-1}})$. We can then find a condition $c \leq_{\B} b$ and a compact 
    set $K \subseteq X_j$ such that $V^{\B} \not\models_c \varphi(\dot{x}_{f_j}, 
    \ldots, \dot{x}_{f_{n-1}})$ and $c \Vdash_{\B} \dot{y} \in \hat{K}$. 
    We can therefore find a function $g \in C(S,K)$ such that 
    $c \Vdash_{\B} \dot{y} = [(\dot{F}_g, \hat{K})] = \dot{x}_g$.
    Then $V^{\B} \not\models_c \psi(\dot{x}_g, \dot{x}_{f_j}, \ldots, 
    \dot{x}_{f_{n-1}})$, so by the induction hypothesis we have $C(S,X) 
    \not\models_c \psi(g, f_j, \ldots, f_{n-1})$, contradicting the assumption 
    that $C(S,X) \models_b \varphi(f_j, \ldots, f_{n-1})$.

    Suppose next that $\varphi$ is of the form $\exists x \psi$ for some 
    $\psi \in \mc L_{j'}(\vec{X})$. Then we can fix $g \in C(S, X_{j'})$ 
    such that $C(S, \vec{X}) \models_b \psi(g, f_j, \ldots, f_{n-1})$. 
    By the induction hypothesis, it follows that $V^{\B} \models 
    \psi(\dot{x}_g, \dot{x}_{f_j}, \ldots, \dot{x}_{f_{n-1}})$, so 
    $V^{\B} \models \varphi(\dot{x}_{f_j}, \ldots, \dot{x}_{f_{n-1}})$. 
    The case in which $\varphi$ is of the form $\exists x \in K \psi$ 
    is the same.

    To verify (2), suppose that $\varphi$ is bounded and 
    $V^{\B} \models_b \varphi(\dot{x}_{f_j}, \ldots, \dot{x}_{f_{n-1}})$. 
    Assume first that $\varphi$ is of the form $\forall x \psi$ for some 
    (bounded) $\psi \in \mc L_{j'}(\vec{X})$. Then, for all $g \in 
    C(S, X_{j'})$, we have $V^{\B} \models_b \psi(\dot{x}_g, \dot{x}_{f_j}, 
    \ldots, \dot{x}_{f_{n-1}})$, so, by the induction hypothesis, we have 
    $C(S,\vec{X}) \models_b \psi(g, f_j, \ldots, f_{n-1})$. Since $g$ was 
    arbitrary, we have $C(S,\vec{X}) \models_b \varphi(f_j, \ldots, f_{n-1})$.

    Finally, assume that $\varphi$ is of the form $\exists x \in K \psi$ 
    for some compact $K \subseteq X_{j'}$ and some (bounded) $\psi \in 
    \mc L_{j'}(\vec{X})$. Then we can find a $\B$-name $\dot{y}$ for 
    an element of $\hat{K}$ such that $V^{\B} \models_b \psi(\dot{y}, 
    \dot{x}_{f_j}, \ldots, \dot{x}_{f_{n-1}})$. We can then find a function 
    $g \in C(S, K)$ such that $b \Vdash_{\B} \dot{y} = \dot{x}_g$. 
    Then $V^{\B} \models_b \psi(\dot{x}_g, \dot{x}_{f_j}, \ldots, 
    \dot{x}_{f_{n-1}})$, so, by the induction hypothesis, we have 
    $C(S, \vec{X}) \models_b \psi(g, f_j, \ldots, f_{n-1})$. Since 
    $g \in C(S,K)$, it follows that $C(S, \vec{X}) \models_b \varphi(f_j, 
    \ldots, f_{n-1})$, as desired.
\end{proof}

We now finally return to Whitehead's problem, where the work of this and 
the previous section allows us to establish the following general result.

\begin{theorem} \label{forcing_thm}
    Suppose that $A$ is an abelian group and $0 \ra K \xrightarrow{\subseteq} F \ra A \ra 0$
    is a free resolution of $A$. Let $B_K$ be a basis for $K$, and let $e:B_K \ra \omega$ be a function. 
    Suppose that $\B$ is a complete Boolean algebra that forces the 
    existence of a homomorphism $\tau : K \ra \Z$ such that 
    \begin{itemize}
        \item in $V^{\B}$, $\tau$ does not lift to a homomorphism 
        $\nu : F \ra \Z$;
        \item for all $t \in B_K$, we have $|\tau(t)| \leq e(t)$.
    \end{itemize}
    Then, in $V$, we have $\iExt^1_{\CondAb}(\ul{A}, \ul{\Z})(S(\B)) 
    \neq 0$.
\end{theorem}

\begin{proof}
    Let $B_F$ be a basis for $F$, and note that elements of $\Hom(K,\Z)$ 
    and $\Hom(F,\Z)$ can be identified with elements of ${^{B_K}}\Z$ and 
    ${^{B_F}}\Z$, respectively, and the restriction map from $\Hom(F,\Z)$ 
    to $\Hom(K,\Z)$ induces a continuous map $\pi : {^{B_F}}\Z \ra {^{B_K}}\Z$.
    Note that the assertion that $A$ is not Whitehead is 
    precisely the assertion that $\pi$ is not surjective, i.e., it corresponds to 
    the sentence $\exists x \in {^{B_K}}\Z ~ \forall y \in {^{B_F}}\Z ~ 
    \pi(y) \neq x$. Noting that the set $U := \{(y,x) \in {^{B_F}}\Z \times 
    {^{B_K}}\Z \mid \pi(y) \neq x\}$ is an open subset of ${^{B_F}}\Z 
    \times {^{B_K}}\Z$, we see that this sentence is an element of 
    $\mc L(({^{B_F}}\Z, {^{B_K}}\Z))$.

    By assumption, $A$ is not Whitehead in $V^{\B}$, and in fact more 
    is true. Let $\hat{\pi}$ be the extension of $\pi$ to 
    $V^{\B}$ induced by the restriction map from $\Hom(F,\Z)$ 
    to $\Hom(K,\Z)$. Then $\B$ forces that there is an element of 
    $\prod_{t \in B_K} [-e(t),e(t)]$  that is not in the range of $\pi$. 
    Let $X_K$ and $X_F$ 
    denote ${^{B_K}}\Z$ and ${^{B_F}}\Z$, respectively, and note that 
    $K^* := \prod_{t \in B_K} [-e(t),e(t)]$ is a compact subset 
    of $X_K$. Let $\varphi$ be the (bounded) sentence 
    $\exists x \in K^* \forall y \in X_F ~ \pi(y) \neq x$. 
    It is readily verified that the following facts hold in $V^{\B}$.
    \begin{itemize}
        \item $(\hat{X}_K)_{\mc K}$ is (homeomorphic to) the set of all 
        $x \in ({^{B_K}}\Z)^{V^{\B}}$ such that there exists 
        $e : B_K \ra \omega$ in $V$ such that $|x(t)| \leq e(t)$ 
        for all $t \in B_K$. A similar characterization holds 
        for $(\hat{X}_F)_{\mc K}$.
        \item $\hat{U}_{\mc K} = \{(y,x) \in (\hat{X}_K)_{\mc K} 
        \times (\hat{X}_F)_{\mc K} \mid \hat{\pi}(y) \neq x\}$.
    \end{itemize}
    By assumption, we then have $V^{\B} \models \varphi$. By 
    Theorem \ref{equivalence_thm}, it follows that 
    $C(S(\B),X_F \times X_K) \models \varphi$, and in fact 
    $C(S(\B), X_F \times K^*) \models \varphi$, i.e., there is a continuous 
    function $f:S(\B) \ra K^*$ such that, for all continuous 
    $g:S(\B) \ra X_F$, the set of $s \in S(\B)$ such that $\pi(g(s)) \neq f(s)$ 
    is dense in $S(\B)$. But this immediately implies that 
    $\iExt^1_{\CondAb}(\ul{A}, \ul{\Z})(S(\B)) \neq 0$, completing the proof 
    of the theorem.
\end{proof}

We now immediately obtain the following corollary, establishing clause (2) 
of Theorem \ref{main_thm}.

\begin{corollary} \label{forcing_cor}
    Suppose that $A$ is a nonfree abelian group and $\kappa$ is the 
    least cardinality of a nonfree subgroup of $A$. Let 
    $S$ be the Stone space of the Boolean completion of the forcing to 
    add $\kappa$-many Cohen reals. Then $\iExt^1_{\CondAb}(\ul{A}, \ul{\Z})(S) 
    \neq \emptyset$.
\end{corollary}

\begin{proof}
    Let $\B$ denote the Boolean completion of the forcing to add 
    $\kappa$-many Cohen reals, let $0 \ra K \xrightarrow{\subseteq} F \ra A \ra 0$ 
    be a free resolution of $A$, and let $B_K$ be a basis for $K$. 
    By Theorem \ref{thm:cohen}, $\B$ forces the existence of a 
    homomorphism $\tau:K \ra \Z$ such that
    \begin{itemize}
        \item in $V^{\B}$, $\tau$ does not lift to a homomorphism 
        $\nu : F \ra \Z$; and
        \item for all $t \in B_K$, $\tau(t) \in \{0,1\}$.
    \end{itemize}
    Therefore, Theorem \ref{forcing_thm} implies that 
    $\iExt^1_{\CondAb}(\ul{A}, \ul{\Z})(S) \neq 0$, as desired.
\end{proof}

We end this section with some observations regarding forcing extensions 
in which Whitehead's problem has a negative answer. We first note 
that the converse of Theorem \ref{forcing_thm} is not true in general; 
namely, if $A$ is an abelian group and $\B$ is a complete 
Boolean algebra, then it does not 
follow from $\iExt^1_{\CondAb}(\ul{A}, \ul{\Z})(S(\B)) \neq 0$ 
that $A$ is not Whitehead in $V^{\B}$. To see this, first recall 
that an abelian group $A$ is $\aleph_1$-free if all of its 
countable subgroups are free, and a subgroup $B$ of an $\aleph_1$-free 
abelian group $A$ is said to be \emph{$\aleph_1$-pure} if $A/B$ is 
$\aleph_1$-free. An abelian group $A$ satisfies \emph{Chase's 
condition} if $A$ is $\aleph_1$-free and every countable 
subgroup of $A$ is contained in a countable $\aleph_1$-pure 
subgroup of $A$. In \cite{griffith}, Griffith showed that one can 
construct in $\ZFC$ an abelian group of size $\aleph_1$ that is 
not free but satisfies Chase's condition. One half of Shelah's 
independence result can then be stated as follows.

\begin{theorem}[Shelah, {\cite{Shelah_infinite_74}}, cf.\ {\cite{eklof}}]
    \label{shelah_ma_thm}
    Suppose that $\MA_{\omega_1}$ holds and $A$ is an abelian 
    group of size $\aleph_1$ that satisfies Chase's condition. 
    Then $A$ is Whitehead.
\end{theorem}

Now suppose that $V = L$, and let $A$ be an abelian group of 
size $\aleph_1$ that satisfies Chase's condition but is not free. 
For concreteness, suppose it is constructed as in 
\cite[Theorem 7.3]{eklof}. It is readily verified that this 
construction is absolute between $V$ and outer models with the 
same $\aleph_1$; in particular, in any ccc forcing extension, it 
remains true that $A$ satisfies Chase's condition and is not free. 
Let $\B$ be the Boolean completion of the standard ccc forcing 
iteration to force $\MA_{\omega_1}$, and let $0 \ra K \ra F 
\ra A \ra 0$ be a free resolution of $A$. Since 
$V = L$, $A$ is not Whitehead in $V$, so there is a map
$\tau \in \Hom(K,\Z)$ that does not lift to a map 
$\sigma \in \Hom(F,\Z)$. Then the constant map $\varphi : S(\B) 
\rightarrow \Hom(K,\Z)$ taking value $\tau$ witnesses that 
$\iExt^1_{\CondAb}(\ul{A}, \ul{\Z})(S(\B)) \neq 0$. However, 
since $A$ still satisfies Chase's condition in $V^{\B}$ and 
$\MA_{\omega_1}$ holds there, it follows that $\B$ forces that 
$A$ is Whitehead.

We can, though, recover an informative partial converse to 
Theorem \ref{forcing_thm}, in the following sense. Note that, 
in the situation described in the paragraph above, forcing with 
$\B$ necessarily adds a new element $\sigma \in \Hom(F,\Z)$ 
that extends $\tau$. However, this new homomorphism $\sigma$ 
must be quite far from any ground model homomorphism; in particular, 
it cannot be bounded by any ground model function when restricted 
to any basis for $F$ that lies in $V$. More generally, we have the 
following result.

\begin{theorem}
    Suppose that $A$ is a non-Whitehead abelian group, 
    $0 \ra K \ra F \ra A \ra 0$ is a free resolution of $A$, 
    and $\tau \in \Hom(K,\Z)$ is a homomorphism that does not 
    extend to a homomorphism $\sigma \in \Hom(F,\Z)$. Suppose 
    moreover that $B_F$ is a~basis for $F$, $e:B_F \ra \omega$ 
    is a function, and $\B$ is a complete Boolean algebra. Then, 
    in $V^{\B}$, there is no homomorphism $\sigma \in 
    \Hom(F,\Z)$ such that
    \begin{itemize}
        \item $\sigma \restriction K = \tau$; and
        \item for all $y \in B_F$, $|\sigma(y)| \leq e(y)$.
    \end{itemize}
\end{theorem}

\begin{proof}
    Let $B_K$ be a basis for $K$, let $X_K := {^{B_K}}\Z$ 
    and $X_F := {^{B_F}}\Z$, let $\pi : 
    X_F \ra X_K$ be the continuous map from 
    the proof of Theorem \ref{forcing_thm}, and again let 
    $U := \{(y,x) \in X_F \times X_K \mid 
    \pi(y) \neq x\}$. Let $\varphi(x)$ be the 
    formula $\forall y \in X_F (y,x) \in U$, 
    where $x$ is a~free variable. Let $f:S(\B) \ra X_K$ be the constant function 
    taking value $\tau \restriction B_K$. Then, by the choice of $\tau$, 
    we have $C(S(\B),X_F \times X_K) \models 
    \varphi(f)$. By Theorem \ref{equivalence_thm}, it follows 
    that $V^{\B} \models \varphi(\dot{x}_f)$. Since $f$ is a constant 
    function, it is readily verified that $\dot{x}_f$ is forced to be 
    equal to the value of $f$, i.e., to $\tau \restriction B_K$. 
    Therefore, it follows that, in $V^{\B}$, for all 
    $\sigma \in \Hom(F,\Z)$, if $\sigma \restriction B_F \in 
    (\hat{X}_F)_{\mc K}$, then $\sigma \restriction K \neq \tau$. 
    But $\sigma \restriction B_F \in (\hat{X}_F)_{\mc K}$ is 
    equivalent to the existence of an $e:B_F \ra \omega$ in $V$ 
    such that $|\sigma(y)| \leq e(y)$ for all $y \in B_F$, so 
    the theorem follows.
\end{proof}

\section{Martin's Axiom and $\CondAb_{\omega_1}$} \label{ma_sec}

In this section, we show that the independence of Whitehead's problem 
in the classical setting can persist if we restrict ourselves 
to $\CondAb_\kappa$ for certain cardinals $\kappa$. 
Both for concreteness and because the category 
$\CondAb_{\omega_1}$ of light condensed sets provides the setting for 
\cite{analytic_stacks}, we focus on the case $\kappa = \omega_1$.
The arguments in this section will be almost entirely set theoretic and 
combinatorial. We present a more algebraic path to related results 
in Section \ref{alternate_section}.

For a given abelian group $A$, if $\ul{A}$ is Whitehead in 
$\CondAb_{\omega_1}$, then $A$ is Whitehead in the classical sense, 
since 
\[
  \iExt^1_{\CondAb_{\omega_1}}(\ul{A}, \ul{\Z})(\ast) = 
  \Ext^1_{\Ab}(A, \Z).
\]
Therefore, one side of the independence of Whitehead's problem 
in $\CondAb_{\omega_1}$ is clear: if it is the case that, classically, 
every Whitehead abelian group is free, then it follows that, for 
every abelian group $A$, if $\ul{A}$ is Whitehead in 
$\CondAb_{\omega_1}$, then $A$ is free. In this section, we establish 
the nontrivial side of this independence. Namely, we prove the 
consistency of the existence of a nonfree abelian group $A$ such 
that $\ul{A}$ is Whitehead in $\CondAb_{\omega_1}$. In fact, 
we shall prove that this follows from $\MA_{\omega_1}$, the same 
hypothesis used by Shelah in his proof of Theorem \ref{shelah_thm}.

In what follows, when we write, e.g., that $S = \varprojlim_{i \in \Lambda} S_i \in \mathsf{Prof}$, 
it should be implicitly understood that $\Lambda$ is a directed partial order and 
each $S_i$ is a finite (discrete) space. For each $i \in \Lambda$, we get a projection 
map $\pi_i:S \ra S_i$. The basic open subsets of $S$ are of the form 
$\pi_i^{-1}\{\bar{s}\}$, where $i \in \Lambda$ and $\bar{s} \in S_i$.
If we write $S = \varprojlim_{i \in \Lambda} S_i \in \mathsf{Prof}_{\omega_1}$, then 
this means, moreover, that $\Lambda$ is countable.

\begin{proposition} \label{basic_clopen_prop}
    Suppose that $S = \varprojlim_{i \in \Lambda} S_i \in \mathsf{Prof}$ and 
    $\{O_k \mid k < n\}$ is a finite partition of $S$ into nonempty clopen 
    sets. Then there is an $i \in \Lambda$ and a function $f:S_i \ra n$ 
    such that, for all $k < n$, we have $O_k = \{s \in S \mid f(\pi_i(s)) = k\}$.
\end{proposition}

\begin{proof}
    This follows immediately from the fact that each $O_k$ is compact and open and can 
    thus be written as a finite union of basic open subsets of $S$, together with 
    the directedness of $\Lambda$.
\end{proof}

\begin{proposition} \label{local_prop}
    Suppose that $S = \varprojlim_{i \in \Lambda} S_i \in \mathsf{Prof}$, $A$ and $B$ are 
    (discrete) abelian groups, and $\varphi:S \ra \Hom(A,B)$ is a continuous map. 
    For every finitely generated subgroup $A_0$ of $A$, there is an $i \in \Lambda$ and a 
    map $\varphi_0:S_i \ra \Hom(A_0,B)$ such that, for all $s \in S$, we have
    \[
      \varphi(s) \restriction A_0 = \varphi_0(\pi_i(s)).
    \]
\end{proposition}

\begin{proof}
    Let $X_0$ be a finite subset of $A$, and let $A_0$ be the subgroup of $A$ 
    generated by $X_0$. For each $\psi:X_0 \ra B$, the set $S_\psi := 
    \{s \in S \mid \varphi(s) \restriction X_0 = \psi\}$ is a clopen subset of 
    $S$, by the continuity of $\varphi$. By the compactness of $S$, there are 
    only finitely many $\psi$ for which $S_\psi$ is nonempty. Therefore, by 
    Proposition \ref{basic_clopen_prop}, there is $i \in \Lambda$ and 
    $f:S_i \ra {^{X_0}}B$ such that, for all $s \in S$, we have 
    $\varphi(s) \restriction X_0 = f(\pi_i(s))$.

    Since $A_0$ is generated by $X_0$, every element of $\Hom(A_0,B)$ 
    is determined by its restriction to $X_0$. Therefore, for each $\bar{s} 
    \in S_i$, there is a unique extension $\varphi_0(\bar{s})$ of 
    $f(\bar{s})$ to an element of $\Hom(A_0,B)$.\footnote{Such an extension 
    must exist because, if $s \in S$ is such that $\pi_i(s) = \bar{s}$, 
    then $\varphi(s) \restriction A_0$ is such an extension; the uniqueness 
    then follows from the previous sentence.} It then follows that, for all 
    $s \in S$, we have $\varphi(s) \restriction A_0 = \varphi_0(\pi_i(s))$, 
    as desired.
\end{proof}

Recall the definition of Chase's condition from the end of 
Section \ref{compact_sec}. We now show that Theorem 
\ref{shelah_ma_thm} transfers to the richer setting 
of $\CondAb_{\omega_1}$.

\begin{theorem}
    Suppose that $\MA_{\omega_1}$ holds and $A$ is an abelian group of cardinality 
    $\aleph_1$ satisfying Chase's condition. Then 
    \[
      \iExt^1_{\CondAb_{\omega_1}}(\ul{A}, \ul{\Z}) = 0.
    \]
\end{theorem}

\begin{proof}
    Let $0 \ra K \ra F \ra A \ra 0$ be the canonical free resolution of $A$,
    let $S = \varprojlim_{i \in \Lambda} S_i \in \mathsf{Prof}_{\omega_1}$, and let 
    $\varphi : S \ra \Hom(K,\Z)$ be continuous. We will produce a continuous 
    map $\psi : S \ra \Hom(F,\Z)$ such that $\psi(s) \restriction K = 
    \varphi(s)$ for all $s \in S$. This will show that the map 
    $\ul{\mathrm{Hom}}_{\CondAb_{\omega_1}}(\ul{F},\ul{\bb{Z}}) 
    \ra \ul{\mathrm{Hom}}_{\CondAb_{\omega_1}}(\ul{K}, 
    \ul{\bb{Z}})$ is surjective, and hence that 
    $\iExt^1_{\CondAb_{\omega_1}}(\ul{A}, \ul{\Z}) = 0$.
    
    Fix a basis 
    $\{\bar{a} \mid a \in A\}$ for $F$; the surjection from $F$ to $A$ 
    in the free resolution above is then generated by the map that sends 
    $\bar{a}$ to $a$ for all $a \in A$. For each subgroup $A' \subseteq A$, 
    let $F_{A'}$ be the free group generated by $\{\bar{a} \mid a \in A'\}$, and 
    let $K_{A'}$ be the kernel of the canonical surjection from $F_{A'}$ to 
    $A'$.

    \begin{claim} \label{extension_claim}
        Suppose that $A'$ is a free subgroup of $A$, with a basis $B'$, and 
        let $f: S \ra {^{B'}}\Z$ be a continuous function (where 
        $\Z$ is discrete and ${^{B'}}\Z$ is given the product topology). 
        Then there is a unique continuous map $\psi':S \ra \Hom(F_{A'}, 
        \Z)$ such that
        \begin{enumerate}
            \item for all $s \in S$ and $b \in B'$, $\psi'(s)(\bar{b}) = 
            f(s)(b)$;
            \item for all $s \in S$, $\psi'(s) \restriction K_{A'} = 
            \varphi(s) \restriction K_{A'}$.
        \end{enumerate}
    \end{claim}

    \begin{proof}
        Fix $s \in S$. To fulfill requirement (1), we are obliged to 
        let $\psi'(s)(\bar{b}) := f(s)(b)$ for all $b \in B'$.
        Now fix $a \in A' \setminus B'$. There is a unique way 
        to express $a$ in the form $\sum_{b \in B^*} c_b b$, where 
        $B^* \subseteq B'$ is finite and, for all $b \in B^*$, 
        $c_b$ is a nonzero integer. To satisfy requirement (2), we are 
        obliged to let
        \[
          \psi'(s)(\bar{a}) := \varphi(s)\left( \bar{a} - 
          \sum_{b \in B^*} c_b \bar{b} \right) + 
          \sum_{b \in B^*} c_b f(s)(b).
        \]
        We must then extend $\psi'(s)$ linearly to the rest of $F_{A'}$. It 
        is routine to verify that $\psi'(s)$ thusly defined is in 
        $\Hom(F_{A'},\Z)$ satisfying (1) and (2), and the continuity of 
        $\psi$ follows immediately from the continuity of $\varphi$ and 
        $f$. The uniqueness of $\psi'$ is evident from the fact that, 
        as the above construction makes clear, all values of $\psi'(s)$ 
        were entirely determined by requirements (1) and (2) and the 
        necessity for $\psi'(s)$ to be a homomorphism.
    \end{proof}

    We define a forcing notion $\P$ to which we will apply $\MA_{\omega_1}$. 
    Conditions in $\P$ are all pairs of the form $p = (A_p, \psi_p)$ such 
    that 
    \begin{itemize}
        \item $A_p$ is a finitely-generated pure subgroup of $A$;
        \item $\psi_p : S \rightarrow \Hom(F_{A_p},\Z)$ is continuous;
        \item for all $s \in S$, $\psi_p(s) \restriction K_{A_p} = 
        \varphi(s) \restriction K_{A_p}$.
    \end{itemize}
    If $p,q \in \P$, then $q \leq p$ if and only if $A_q \supseteq A_p$ 
    and, for all $s \in S$, $\psi_q(s) \restriction F_{A_p} = 
    \psi_p(s)$.

    \begin{claim} \label{dense_claim}
        For all $a \in A$, the set $D_a := \{p \in \P \mid a \in A_p\}$ 
        is dense in $\P$.
    \end{claim}

    \begin{proof}
        Fix $a \in A$ and $p \in \P$. We will find $q \leq p$ with 
        $q \in D_a$. Since $A$ is $\aleph_1$-free, we 
        can find a finitely generated pure subgroup $A'$ of $A$ such 
        that $A_p \cup \{a\} \subseteq A'$. Then $A'/A_p$ is free, so 
        we can fix a basis for $A'$ of the form $B_0 \cup B_1$, where 
        $B_0$ is a basis for $A_p$. Define a function $f:S \ra 
        {^{B_0 \cup B_1}}\Z$ as follows: for all $s \in S$ and 
        $b \in b_0 \cup B_1$, let
        \begin{align*}
            f(s)(b) := \begin{cases}
                \psi_p(s)(\bar{b}) & \text{if } b \in B_0 \\ 
                0 & \text{if } b \in B_1.
            \end{cases}
        \end{align*}
        Since $\psi_p$ is continuous, it follows that $f$ is continuous.
        Let $\psi':S \ra \Hom(F_{A'},\Z)$ be the unique continuous map 
        given by Claim \ref{extension_claim} applied to $A'$, 
        $B_0 \cup B_1$, and $f$. Then $q := (A', \psi') \in D_a$. 
        Moreover, the uniqueness clause of Claim \ref{extension_claim} 
        together with the definition of $f$ implies that 
        $\psi'(s) \restriction F_{A_p} = \psi_p(s)$ for all $s \in S$, 
        so we have $q \leq p$.
    \end{proof}

    \begin{claim}
        $\P$ has the ccc.
    \end{claim}

    \begin{proof}
        Let $\vec{p} = \langle p_\eta \mid \eta < \omega_1 \rangle$ 
        be a sequence of conditions in $\P$. By (the proof of) 
        \cite[Lemma 7.5]{eklof}, by thinning out $\vec{p}$ to some 
        cofinal subsequence if necessary, 
        we can assume that there is a free subgroup $A'$ of $A$ that is 
        pure in $A$ and such that $A_{p_\eta} \subseteq A'$ for all 
        $\eta < \omega_1$ (this is the only use of the full power of 
        Chase's condition in the proof). Let $B'$ be a basis for $A'$. 
        Using Claim \ref{dense_claim} and extending each $p_\eta$ if 
        necessary, we can assume that, for all $\eta < \omega_1$, 
        $A_{p_\eta}$ is freely generated by some finite subset 
        $B'_\eta$ of $B'$. By thinning out $\vec{p}$ again if necessary, 
        we can assume that the sets $\{B'_\eta \mid \eta < \omega_1\}$ 
        form a $\Delta$-system, with root $R$.

        For all $\eta < \omega_1$, define a map $f_\eta:S \ra 
        {^{B'_\eta}}\Z$ by letting $f_\eta(s)(b) := \psi_{p_\eta}(s)(\bar{b})$ 
        for all $s \in S$ and $b \in B'_\eta$. By Proposition 
        \ref{local_prop}, we can fix $i_\eta \in \Lambda$ and a map 
        $g_\eta:S_{i_\eta} \ra {^R}\Z$ such that, for all $s \in S$ and $b \in R$, 
        we have $f_\eta(s)(b) = g_\eta(\pi_i(s))(b)$. Since there are only 
        countably many choices for $i_\eta$ and $g_\eta$, we can find a 
        fixed $i$ and $g$ and ordinals $\eta < \xi < \omega_1$ such that 
        $i_\eta = i_\xi = i$ and $g_\eta = g_\xi = g$. Note that, for all 
        $s \in S$ and $b \in R$, this implies that $f_\eta(s)(b) = 
        f_\xi(s)(b)$.

        Let $B^* := B_\eta \cup B_\xi$, and let $A^*$ be the free group 
        generated by $B^*$. Then $A^*$ is a pure subgroup of $A'$; since 
        $A'$ is pure in $A$, this implies that $A^*$ is also pure in $A$.
        Define a function $f:S \ra {^{B^*}}\Z$ by letting $f(s) = 
        f_\eta(s) \cup f_\xi(s)$ for all $s \in S$. Since $\psi_{p_\eta}$ 
        and $\psi_{p_\xi}$ are continuous (and hence $f_\eta$ and $f_\xi$ 
        are continuous), it follows that $f$ is continuous. Let $\psi^* : 
        S \ra \Hom(F_{A^*},\Z)$ be the unique continuous map obtained by 
        applying Claim \ref{extension_claim} to $A^*$, $B^*$, 
        and $f$. Then $q = (A^*, \psi^*) \in \P$ and, by the uniqueness 
        clause of Claim \ref{extension_claim}, it extends both $p_\eta$ 
        and $p_\xi$. Therefore, $\vec{p}$ does not enumerate an antichain, 
        and hence $\P$ has the ccc.
    \end{proof}
    Now apply $\MA_{\omega_1}$ to obtain a filter $G \subseteq \P$ such that, 
    for all $a \in A$, we have $G \cap D_a \neq \emptyset$. Define 
    $\psi:S \ra \Hom(F,\Z)$ by letting $\psi(s) = \bigcup \{\psi_p(s) \mid 
    p \in G\}$. The fact that each $\psi_p$ is a continuous map such 
    that, for all $s \in S$, we have $\psi_p(s) \restriction K_{A_p} 
    = \varphi(s) \restriction K_{A_p}$ implies that $\psi$ is a continuous 
    map such that, for all $s \in S$, we have $\psi(s) \restriction K 
    = \varphi(s)$, as desired.
\end{proof}

Recalling Griffith's result \cite{griffith} stating that there is provably in $\ZFC$ a 
nonfree abelian group of size $\aleph_1$ satisfying Chase's condition, 
we thus obtain the following corollary, establishing Theorem 
\ref{local_theorem}(1).

\begin{corollary} \label{ma_cor}
    If $\MA_{\omega_1}$ holds, then there is a nonfree abelian 
    group $A$ of size $\aleph_1$ such that $\ul{A}$ is Whitehead in 
    $\CondAb_{\omega_1}$. 
\end{corollary}

\section{A different vantage point} \label{alternate_section}

In this section, we describe and apply a lemma providing an alternative, 
more purely algebraic path to several of our results, as well as to a few new ones.
Recall from Section \ref{top_spaces_sec} that if $\B$ is the complete Boolean algebra of regular open 
subsets of an extremally disconnected compact Hausdorff space $S$, then $C(S,\Z)$ is precisely the group of $\B$-names 
for integers. 

\begin{lemma} \label{ext_calc_lemma}
For all abelian groups $A$ and extremally disconnected profinite sets~$S$,
$$\underline{\mathrm{Ext}}^1_{\mathsf{CondAb}}(\underline{A},\underline{\mathbb{Z}})(S) = \mathrm{Ext}^1_{\mathsf{Ab}}(A,C(S,\mathbb{Z})).$$
\end{lemma}

\begin{proof}
The chain of equalities is as follows:
\begin{align*}
\underline{\mathrm{Ext}}^1_{\mathsf{CondAb}}(\underline{A},\underline{\mathbb{Z}})(S) & = \mathrm{Ext}^1_{\mathsf{CondAb}}(\underline{A}\otimes\mathbb{Z}[\underline{S}],\underline{\mathbb{Z}})  \\
& = \mathrm{Ext}^1_{\mathsf{CondAb}}(\underline{A},\underline{\mathrm{Hom}}_{\mathsf{CondAb}}(\mathbb{Z}[\underline{S}],\underline{\mathbb{Z}})) \\
& = \mathrm{Ext}^1_{\mathsf{CondAb}}(\underline{A},\underline{\mathrm{Hom}}_{\mathsf{Cond}}(\underline{S},\underline{\mathbb{Z}})) \\
& = \mathrm{Ext}^1_{\mathsf{CondAb}}(\underline{A},\underline{C(S,\mathbb{Z}})) \\
& = \mathrm{Ext}^1_{\mathsf{Ab}}(A,C(S,\mathbb{Z}))
\end{align*}
To see the fifth equality, apply Lemma \ref{lemma:evaluation}; for the third and fourth, respectively, apply \cite[Prop.\ 3.5 and Prop.\ 3.2]{Aparicio_condensed_21}. The first equality follows from \cite[p.\ 13 (iii) (cf.\ p.\ 25)]{CS1}. This leaves us only with the second. Here the idea, evidently, is to apply the Hom-tensor adjunction at the level of $\mathrm{Ext}^1_{\mathsf{CondAb}}$, and we may do so by the naturality of the adjunction at the $\mathrm{Hom}_{\mathsf{CondAb}}$-level, the flatness of $\mathbb{Z}[\underline{S}]$ \cite[p.\ 13]{CS1}, and the fact, noted already, that $\ul{\bigoplus_I\Z}=\bigoplus_I\ul{\Z}$ for any set $I$.
More particularly, letting $\underline{\mathcal{P}}$, much as in the proof of Lemma \ref{lemma:evaluation}, denote the condensation of a free resolution of $\underline{A}$, we have:
\begin{align*}
\mathrm{Ext}^1_{\mathsf{CondAb}}(\underline{A}\otimes\mathbb{Z}[\underline{S}],\underline{\mathbb{Z}}) & = \mathrm{H}^1(\mathrm{Hom}_{\mathsf{CondAb}}(\underline{\mathcal{P}}\otimes\mathbb{Z}[\underline{S}],\underline{\mathbb{Z}})) \\
& = \mathrm{H}^1(\mathrm{Hom}_{\mathsf{CondAb}}(\underline{\mathcal{P}},\underline{\mathrm{Hom}}_{\mathsf{CondAb}}(\mathbb{Z}[\underline{S}],\underline{\mathbb{Z}})) \\
& = \mathrm{Ext}^1_{\mathsf{CondAb}}(\underline{A},\underline{\mathrm{Hom}}_{\mathsf{CondAb}}(\mathbb{Z}[\underline{S}],\underline{\mathbb{Z}})),
\end{align*}
and this chain of equalities, if sound, will complete the argument.
Let us check the details.
Observe first that $\underline{\mathcal{P}}$ is a projective resolution of $\ul{A}$, by \cite[Prop. 2.18]{Aparicio_condensed_21} and Proposition \ref{prop:proj} below; this shows the third equality.
The second is the aforementioned adjunction.
For the first, we have noted already that $\Z[\ul{S}]$ is flat, so the question reduces to that of whether the terms of $\underline{\mathcal{P}}\otimes\Z[\ul{S}]$ are projective.
For this observe simply that
$\big(\ul{\bigoplus_I\Z}\big)\otimes\Z[\ul{S}]=\big(\bigoplus_I\ul{\Z}\big)\otimes\Z[\ul{S}]=\bigoplus_I\big(\ul{\Z}\otimes\Z[\ul{S}]\big)=\bigoplus_I\Z[\ul{S}]$ for any set $I$; that such sums are projective follows from Proposition \ref{prop:proj} below.
\end{proof}
Note that the lemma allows us to recast this paper's shaping background in uniformly classical terms: Shelah's result was that the implication
$$\big[\,\mathrm{Ext}^1_{\mathsf{Ab}}(A,C(S,\Z))=0\text{ for }S=*\,\big]\;\Rightarrow\;\text{the abelian group }A\text{ is free}$$
is independent of the $\ZFC$ axioms. Clausen and Scholze's result, in contrast, can now be seen to follow from the fact that the implication
$$\big[\,\mathrm{Ext}^1_{\mathsf{Ab}}(A,C(S,\Z))=0\text{ for all }S\in\mathsf{ED}\,\big]\;\Rightarrow\;\text{the abelian group }A\text{ is free}$$
is indeed a $\ZFC$ theorem.
This implication is, moreover, straightforward to prove, since by a result of N\"{o}beling (\cite{nobeling}, see also \cite[Proposition~XI.4.4]{Eklof_Mekler}), 
generalizing work of Specker \cite{specker}, the group 
$C(S,\mathbb{Z})$ is itself a free abelian group, one isomorphic to the free abelian group on the topological weight of $S$.

In fact, we may be more precise. Observe that, if $A$ is 
a nonfree abelian group of cardinality $\kappa$, then the canonical 
free resolution of $A$ yields $\Ext^1_{\Ab}(A, \Z^{(\kappa)}) \neq 
0$, where $\Z^{(\kappa)}$ denotes the free abelian group on 
$\kappa$-many generators. Moreover, if $A$ is an abelian group and 
$A_0$ is a nonfree \emph{subgroup} of $A$ of cardinality $\kappa$, 
then an application of $\Hom(\cdot, \Z^{(\kappa)})$ to the short exact sequence $0 \ra A_0 \ra A \ra A/A_0 \ra 0$ 
yields a long exact sequence, a portion of which is 
\[
\ldots \ra \Ext^1_{\Ab}(A, \Z^{(\kappa)}) \ra 
\Ext^1_{\Ab}(A_0, \Z^{(\kappa)}) \ra 0.
\]
It follows that $\Ext^1_{\Ab}(A, \Z^{(\kappa)}) \neq 0$. Therefore, 
Lemma \ref{ext_calc_lemma} immediately yields the following corollary, providing an 
alternative proof of clause (2) of Theorem~\ref{main_thm}.

\begin{corollary} \label{cor_11_2}
    Suppose that $A$ is a nonfree abelian group and $\kappa$ is the least 
    cardinality of a nonfree subgroup of $A$. Then
    \[
    \iExt^1_{\CondAb}(\ul{A}, \ul{\Z})(S) \neq 0
    \]
    for every $S \in \mathsf{ED}$ of weight at least $\kappa$. \qed
\end{corollary}

Lemma \ref{ext_calc_lemma} also sheds additional light 
on the results of Section \ref{ma_sec}, entailing both an alternative 
proof of a variation of that section's main result and limits, in the presence of sufficiently large cardinals, on the extent to 
which this result can be generalized.

We first note that, by results in \cite{eklof_shelah}, it is 
consistent with the axioms of $\ZFC$ that, for 
every cardinal $\kappa$, there is a nonfree abelian group $A$ such that 
$\Ext^1_{\Ab}(A,B) = 0$ for every abelian group $B$ of cardinality 
${\leq} \kappa$. Thus, Lemma \ref{ext_calc_lemma} yields the following 
corollary, establishing clause (2) of Theorem \ref{local_theorem} and 
indicating that, consistently, the positive answer to the interpretation of 
Whitehead's problem in $\CondAb$ established by Theorem \ref{cs_whitehead_thm} 
fails in $\CondAb_\kappa$ for \emph{every} uncountable cardinal $\kappa$.

\begin{corollary} \label{up_cor}
    It is consistent with the axioms of $\ZFC$ that, for every 
    cardinal~$\kappa$, there is a nonfree abelian group $A$ such that 
    $\iExt^1_{\CondAb}(\ul{A},\ul{\Z})(S) = 0$ for all $S \in \mathsf{ED}_\kappa$. 
    In particular, it is consistent that, for every uncountable cardinal 
    $\kappa$, there is a~nonfree abelian group $A$ such that $\ul{A}$ is 
    Whitehead in $\CondAb_{\kappa}$.
\end{corollary}

On the other hand, the existence of sufficiently large cardinals will bound the extent to which the phenomena of Corollary \ref{up_cor} can 
occur. In particular, if $\kappa$ is a strongly compact cardinal, then 
$\CondAb_{\kappa}$ is sufficiently rich to 
guarantee a positive answer to the appropriate interpretation of Whitehead's 
problem. The following corollary yields clause (3) of Theorem \ref{local_theorem}.

\begin{corollary} \label{strongly_compact_cor}
    Suppose that $\kappa$ is a strongly compact cardinal and $A$ is an 
    abelian group. If $\ul{A}$ is Whitehead in $\CondAb_\kappa$, then 
    $A$ is free.
\end{corollary}

\begin{proof}
    Suppose that $\ul{A}$ is Whitehead in $\CondAb_\kappa$. In particular, 
    for for every $S \in \mathsf{ED}_\kappa$, we have $\iExt^1_{\CondAb}(
    \ul{A},\ul{\Z})(S) = 0$. By Lemma~\ref{ext_calc_lemma} and the aforementioned 
    result of N\"{o}beling indicating that, for every $S \in \mathsf{ED}$, 
    $C(S,\Z)$ is the free abelian group on the topological weight of $S$, it follows that, 
    for every cardinal $\lambda < \kappa$, we have $\Ext^1_{\Ab}(A, 
    \Z^{(\lambda)}) = 0$. By the observations in the paragraph immediately 
    preceding Corollary~\ref{cor_11_2}, it follows that $A$ is 
    $\kappa$-free, i.e., every subgroup of $A$ of cardinality ${<}\kappa$ 
    is free. By \cite[Corollary II.3.11]{Eklof_Mekler}, since $\kappa$ 
    is strongly compact, $A$ is free, as desired.\footnote{As 
    \cite[Corollary~II.3.11]{Eklof_Mekler} makes clear, we could have merely assumed 
    that $\kappa$ is only $L_{\omega_1\omega}$-compact, a weaker but lesser-known hypothesis than the strong compactness of $\kappa$.}
\end{proof}

\begin{remark}
    In the proof of Corollary \ref{strongly_compact_cor}, 
    the crucial property 
    of the cardinal $\kappa$ was the fact that every $\kappa$-free 
    group is free. This is true if $\kappa$ is a strongly compact 
    cardinal, but we remark that, consistently, this can also be 
    true for smaller values of $\kappa$. In particular, 
    it is shown in \cite[\S 4]{Magidor_Shelah} that, assuming the consistency of the existence of infinitely many supercompact cardinals, there exists a model of $\ZFC$ in which this condition 
    holds for the first cardinal fixed point of the $\aleph$ function, i.e., the least $\kappa$ such that $\kappa = \aleph_\kappa$.
\end{remark}

\section{Locally compact and solid abelian groups} \label{lca_sec}

In this section, we prove Theorem \ref{thm:LCA_intro}, which we restate here for 
the reader's convenience.

\begin{theorem}
\label{thm:LCA}
Let $A$ be an object of the category $\mathsf{LCA}$ of locally compact abelian groups. Then
$$\underline{\mathrm{Ext}}^1_{\CondAb}(\ul{A},\ul{\mathbb{Z}})=0$$
if and only if $A$ is projective in $\mathsf{LCA}$.
\end{theorem}

\begin{proof}
By the structure theorem for locally compact abelian groups (see \cite[Thm.\ 4.1(i)]{CS1}), any locally compact abelian group $A$ is of the form $\mathbb{R}^n\times B$ for some nonnegative integer $n$, where $B$ is an extension of a compact group by a discrete one, or where, in other words, there exists a short exact sequence
\begin{equation}
\label{eq:compact_ses}
0\to C\to B\to D\to 0
\end{equation}
with $C$ compact and $D$ discrete whose image in $\CondAb$ is exact as well (see \cite[Props.\ 4.9, 2.18]{Aparicio_condensed_21}).
Since $\underline{\mathrm{Ext}}^1_{\CondAb}(\ul{\mathbb{R}},\ul{\mathbb{Z}})=0$ \cite[p.\ 25]{CS1}, we have 
\[
\underline{\mathrm{Ext}}^1_{\CondAb}(\ul{A},\ul{\mathbb{Z}})=0 \; \Leftrightarrow \; \underline{\mathrm{Ext}}^1_{\CondAb}(\ul{B},\ul{\mathbb{Z}})=0.
\]
Consider next the following portion of the long exact sequence deriving from equation (\ref{eq:compact_ses}):
$$\cdots\to 0\to\underline{\mathrm{Ext}}^1_{\CondAb}(\ul{D},\ul{\mathbb{Z}})\to\underline{\mathrm{Ext}}^1_{\CondAb}(\ul{B},\ul{\mathbb{Z}})\to\underline{\mathrm{Ext}}^1_{\CondAb}(\ul{C},\ul{\mathbb{Z}})\to 0\to\cdots$$
(The rightmost $0$ follows from \cite[p.\ 27]{CS1} (e.g., from the comment following Cor.\ 4.9); the leftmost $0$ follows from \cite[Prop.\ 4.2]{CS1}.) Hence $$\iExt^1_{\CondAb}(\ul{A},\ul{\Z}) = \underline{\mathrm{Ext}}^1_{\CondAb}(\ul{B},\ul{\mathbb{Z}})=0$$ if and only if
\begin{enumerate}
\item $\underline{\mathrm{Ext}}^1_{\CondAb}(\ul{D},\ul{\mathbb{Z}})=0$ and
\item $\underline{\mathrm{Ext}}^1_{\CondAb}(\ul{C},\ul{\mathbb{Z}})=0$.
\end{enumerate}
By Theorem \ref{cs_whitehead_thm}, item (1) holds if and only if $D=\bigoplus_I\mathbb{Z}$ for some $I$.
Also, the argument of \cite[Prop.\ 4.11]{Aparicio_condensed_21} together with \cite[Prop.\ 4.3(ii)]{CS1} shows that
$\underline{\mathrm{Ext}}^1_{\CondAb}(\ul{C},\ul{\mathbb{R}})=0$; this together with the fact that $\ul{\mathrm{Hom}}(\ul{C},\ul{\mathbb{R}})=\ul{\mathrm{Hom}(C,\mathbb{R})}=0$ (\cite[Prop.\ 4.2]{CS1}) implies that the map $\ul{\mathrm{Hom}}(\ul{C},\ul{\mathbb{R}/\mathbb{Z}})\to\underline{\mathrm{Ext}}^1_{\CondAb}(\ul{C},\ul{\mathbb{Z}})$ appearing in the long exact sequence induced by applying $\ul{\mathrm{Hom}}(\ul{C},-)$ to $\ul{\Z}\to\ul{\R}\to\ul{\R/\Z}$ is an isomorphism. Again applying \cite[Prop.\ 4.2]{CS1}, we have $\underline{\mathrm{Ext}}^1_{\CondAb}(\ul{C},\ul{\mathbb{Z}})=\ul{\mathrm{Hom}(C,\mathbb{R}/\mathbb{Z})}=\ul{\mathbb{D}(C)}$, the condensed image of the Pontryagin dual group of $C$; in particular, item (2) holds if and only if $C=0$. Summing up, $\underline{\mathrm{Ext}}^1_{\CondAb}(\ul{A},\ul{\mathbb{Z}})=0$ if and only if $A$ is of the form
$\mathbb{R}^n\times D$ for some $D=\bigoplus_I\mathbb{Z}$. As noted, by \cite[Prop.\ 3.3]{Moskowitz_homological_67}, such $A$ are precisely the projective objects of $\mathsf{LCA}$.
\end{proof}

This theorem carries as corollary a negative answer to one other condensed version of Whitehead's Problem:
\begin{corollary}
\label{Cor:condensedWhiteheadfails}
As the condensed abelian group $\ul{\mathbb{R}}$ witnesses, $\underline{\mathrm{Ext}}^1_{\CondAb}(A,\ul{\mathbb{Z}})=0$ does not imply that $A$ is projective in $\CondAb$.
\end{corollary}
To see this, recall that, as noted, the compact projective condensed abelian groups of the form $\mathbb{Z}[\ul{S}]$ with $S \in \mathsf{ED}$ generate $\CondAb$ (\cite[p.\ 16--17 and Proposition~3.5]{CS3}); in particular, for any condensed abelian group $X$ there exists a set $I$, a~collection $\{S_i\mid i\in I\}$ of spaces in $\mathsf{ED}$, and an epimorphism $r:\bigoplus_I\mathbb{Z}[\ul{S_i}]\to X$.
If $X$ is projective then $r$ admits a right inverse $s$, hence:
\begin{proposition}
\label{prop:proj}
The projective objects of $\CondAb$ are exactly the retracts of direct sums of groups of the form $\mathbb{Z}[\ul{S}]$ with $S$ extremally disconnected.
\end{proposition}
Hence it will suffice to show that $\mathbb{R}$ is not a retract of such a sum.
This is immediate from the following lemma.
\begin{lemma}
\label{lemma:underlying_set}
The underlying group of $\mathbb{Z}[\ul{S}]$ is the free abelian group $\bigoplus_S\mathbb{Z}$.
\end{lemma}
By the lemma, any section $s:\ul{\mathbb{R}}\to\bigoplus_I\mathbb{Z}[\ul{S_i}]$ of an $r$ as above would evaluate over a point $\ast$ as an injection $s(*):\mathbb{R}\to\bigoplus_I\bigoplus_{S_i}\mathbb{Z}$, but any nonzero element in the image of $s(*)$ would need to be divisible by every positive $n$, a contradiction.
The following, therefore, will complete the proof of Corollary~\ref{Cor:condensedWhiteheadfails}.
\begin{proof}[Proof of Lemma~\ref{lemma:underlying_set}]
As noted, $\mathbb{Z}[\ul{S}]$ is the sheafification of the presheaf
\begin{equation}
\label{eq:Z[S]}T\mapsto \bigoplus_{\mathrm{Cont}(T,S)}\mathbb{Z}
\end{equation}
Recall also that the sheafification $\mathcal{F}^\sharp$ of a presheaf $\mathcal{F}$ on a site is given by the twofold iteration of the $+$ operation, where
$$\mathcal{F}^{+}(U)=\mathrm{colim}_{\mathcal{U}\in\mathrm{cov}(U)}\,\,\mathrm{H}^0(\mathcal{U};\mathcal{F}).$$
Here
$$\mathrm{H}^0(\mathcal{U};\mathcal{F})=\{(s_i)\in\prod_{i\in I} \mathcal{F}(U_i)\mid s_i |_{U_i\times_U U_j} = s_j |_{U_i\times_U U_j}\textnormal{ for all }i,j\in I\}$$
with respect to a~given cover $\mathcal{U}=\{U_i\to U\mid i\in I\}$ (see \cite[\href{https://stacks.math.columbia.edu/tag/00W1}{Tag 00W1}]{stacks-project}), and $\mathcal{U}\prec\mathcal{V}=\{V_j\to U\mid j\in J\}$ in $\mathrm{cov}(U)$ if there exists an $a:J\to I$ and maps $V_j\to U_{a(j)}$ factoring $V_j\to U$ through $U_{a(j)}\to U$ for all $j\in J$.
Observe now that if $U=*$ and $\mathcal{F}$ is the presheaf (\ref{eq:Z[S]}) over the condensed site $*_{\text{pro\'{e}t}}$ then the cover $\mathcal{U}_\mathrm{id}:=\{\mathrm{id}:*\to *\}$ is cofinal in $\mathrm{cov}(U)$;
it follows that $\mathcal{F}^\sharp(*)=\mathcal{F}^{+}(*)=\bigoplus_S\mathbb{Z}$, as claimed.
\end{proof}

Theorem \ref{thm:LCA} opens onto broader questions about alternative interpretations of Whitehead's problem in the condensed setting. These may be thought of as defined by two main parameters: the subscript $s$ of the $\underline{\mathrm{Ext}}^1$ term, and the domain $d$ of $\underline{\mathrm{Ext}}^1_s(\cdot,\underline{\mathbb{Z}})$.
After all, in retrospect, it is not altogether surprising that ($s=\mathsf{CondAb}$)-valued invariants $\underline{\mathrm{Ext}}^1_s$ should be strong enough to settle the problem over the comparatively weaker domains of $d=\mathsf{Ab}$ or even $d=\mathsf{LCA}$ (we hope to have shown that it is altogether \emph{interesting} nevertheless).
A natural next question, then, is how much further this $d$ may be pushed.
By Corollary~\ref{Cor:condensedWhiteheadfails}, of course, Whitehead's problem has a $\ZFC$ solution when $d=\mathsf{CondAb}$ as well, but it is an unsatisfying one, deriving as it does from the arguably spurious case of $\underline{\mathrm{Ext}}^1_{\mathsf{CondAb}}(\underline{\mathbb{R}},\underline{\mathbb{Z}})=0$.
Against this background, particularly interesting is the problem when $d$ equals the subcategory $\mathsf{Solid}$ of $\mathsf{CondAb}$ mentioned briefly in our introduction.
Forming, as it does, the paradigmatic example of an analytic ring, this is one of $\mathsf{CondAb}$'s most critical subcategories (see \cite[\S 5--7]{CS1} and \cite[\S 2]{CS2}); moreover, as noted, it was this category's features which Clausen and Scholze leveraged in their original proof of Theorem~\ref{cs_whitehead_thm}.
The solid setting, being ``very much a `nonarchimedean' notion'' \cite[footnote 10 on p.\ 32]{CS1}, is also one in which $\underline{\mathbb{R}}$ does not appear.\footnote{More precisely, the inclusion $\mathsf{Solid}\to\mathsf{CondAb}$ possesses a left adjoint \emph{solidification} functor notated $M\mapsto M^\blacksquare$, and the solidification $\underline{\mathbb{R}}^\blacksquare$ of $\underline{\mathbb{R}}$ is $0$ \cite[Corollary~6.1(iii)]{CS1}.}
Within it, lastly, $\underline{\mathrm{Ext}}^1_{\mathsf{CondAb}}$ and $\underline{\mathrm{Ext}}^1_{\mathsf{Solid}}$ coincide \cite[Lemma~5.9]{CS1}, so that this instance of the problem recovers the $s=d$ symmetry of the classical one:
\begin{question}
\label{Ques:Solid}
Let $A$ be a solid abelian group. Does $\underline{\mathrm{Ext}}^1_{\mathsf{Solid}}(A,\underline{\mathbb{Z}})=0$ (or equivalently $\underline{\mathrm{Ext}}^1_{\mathsf{CondAb}}(A,\underline{\mathbb{Z}})=0$) imply that $A$ is projective?
\end{question}
We thank Peter Scholze for answering with a multifaceted \emph{no}; the following observations are essentially his.

First, for the solid abelian group $A=\prod_{\mathbb{N}}\bigoplus_{\mathbb{N}}\ul{\mathbb{Z}}$ and any $n\geq 0$,
$$\ul{\mathrm{Ext}}^n_{\mathsf{CondAb}}(A,\ul{\mathbb{Z}})(S)\cong\mathrm{lim}^n\,\mathbf{A}[C(S,\mathbb{Z})],$$
where $\mathbf{A}[H]$ denotes for any abelian group $H$ the inverse system indexed by functions $f:\bb{N}\to\bb{N}$, with terms $A_f:=\bigoplus_{i\in\bb{N}}H^{(f(i))}$
and natural projection maps $A_g\to A_f$ for any $f$ which is pointwise dominated by $g$.
See \cite{blh_inf_comb} for details of this conversion; the immediate point, though, is that these derived limits have been the subject of much recent study (see \cite{SVHDL,SVHDLwLC,VV}), from which the consistency with the $\ZFC$ axioms of the statement that
$$\mathrm{lim}^1\,\mathbf{A}[H]=0\;\text{ but }\mathrm{lim}^2\,\mathbf{A}[H]\neq 0$$
for every nontrivial abelian group $H$, for example, is known. In fact this holds in any model of the Proper Forcing Axiom; whenever it does, $A$ will clearly witness a negative answer to Question \ref{Ques:Solid}.

Left open in the preceding analysis is the possibility of a consistently affirmative answer to Question \ref{Ques:Solid}, since
$$\mathrm{lim}^n\,\mathbf{A}[H]=0\text{ for all }n>0\text{ and abelian groups }H$$
and hence
$$\underline{\mathrm{Ext}}^n_{\mathsf{Solid}}(A,\underline{\mathbb{Z}})=0\text{ for all }n>0$$
is also consistent with the $\ZFC$ axioms \cite[Thm.\ 3]{Bannister}.
By the following argument, though, $A$ witnesses a negative answer to Question \ref{Ques:Solid} in this scenario as well. 

Consider the surjection
$$t:\bigoplus_{\mathbb{N}^\mathbb{N}}\prod_{\mathbb{N}}\mathbb{Z}\to\prod_{\mathbb{N}}\bigoplus_{\mathbb{N}}\mathbb{Z}$$
defined as follows: fix $g:\mathbb{N}\to\mathbb{N}$ and an element $a$ of the $g^{\mathrm{th}}$ summand of the domain of $t$; clearly it will suffice to describe the effect of $t$ on such $a$.
Here it is convenient to adopt the (elsewhere innocuous) convention that $0$ is not a natural number: the first $g(1)$ many entries in $a$ are then mapped to the first $g(1)$ summands of the first factor of $\prod_{\mathbb{N}}\bigoplus_{\mathbb{N}}\mathbb{Z}$, the next $g(2)$ many entries to the first $g(2)$ many summands of the second factor of $\prod_{\mathbb{N}}\bigoplus_{\mathbb{N}}\mathbb{Z}$, and so on.
Equipping $\mathbb{Z}$ with the discrete topology and letting products and sums induce the product and coproduct topologies, respectively, we may view $t$ as a continuous mapping of topological abelian groups.
Consider the condensed image $\ul{t}$ of $t$: observe that its codomain is $A$, its domain is solid, and that it remains an epimorphism (cf.\ \cite[Prop.\ 2.18]{Aparicio_condensed_21}). Thus, to show that $A$ is not projective in $\mathsf{Solid}$, it will suffice to show that $\ul{t}$ does not split, and for this it will suffice to show that the mapping $\ul{t}(*)=t$ of discrete abelian groups described above possesses no right-inverse $s$. By \cite[Cor.\ III.1.9]{Eklof_Mekler}, though, for any
$$s:\prod_{\mathbb{N}}\bigoplus_{\mathbb{N}}\mathbb{Z}\to\bigoplus_{\mathbb{N}^\mathbb{N}}\prod_{\mathbb{N}}\mathbb{Z}$$
there exists an $m\in\mathbb{N}$ and finite $F\subseteq\mathbb{N}^\mathbb{N}$ such that the image of the restriction of $s$ to $\prod_{n\geq m}\bigoplus_{\mathbb{N}}\mathbb{Z}$ is contained within $\bigoplus_{f\in F}\prod_{\mathbb{N}}\mathbb{Z}$.
Letting $k=\max\{f(m)\mid f\in F\}+1$ and writing $x$ for any element of $\prod_{n\geq m}\bigoplus_{\mathbb{N}}\mathbb{Z}$ which is nonzero at the $k^{\mathrm{th}}$ summand of the $m^{\mathrm{th}}$ factor, we must then have $ts(x)\neq x$. This shows that $t$ possesses no right-inverse,\footnote{One can alternatively rule out such a right-inverse via the $\alpha=2$ instance of \cite[Thm.\ X.2.3]{Eklof_Mekler}.} and hence that $A$ is not projective, completing the argument.

Even this doesn't altogether settle the question, though, since it is also consistent with the $\ZFC$ axioms (and follows from the continuum hypothesis) that $\ul{\mathrm{Ext}}^1(A,\ul{\bb{Z}})\neq 0$ \cite[Thm.\ 2]{Mardesic}.
For a \emph{model-independent} negative answer to Question \ref{Ques:Solid}, take the solid abelian group $B=\ul{\bb{Z}_p}/\ul{\bb{Z}[1/q]}$ for any distinct prime numbers $p$ and $q$.
As noted in Clausen's Lecture 11 of \cite{analytic_stacks},
$$\mathrm{Ext}^2_{\mathsf{CondAb}}(B,\ul{\bb{Z}})=\bb{Z}_q/\bb{Z}[1/p],$$ hence $B$ is not projective; it remains only to check that it is Whitehead.
In fact, for the numerator $$\ul{\bb{Z}_p}=\varprojlim\,(\ul{\bb{Z}/p\bb{Z}}\leftarrow\ul{\bb{Z}/p^2\bb{Z}}\leftarrow\cdots)$$
of $B$, we have
\begin{align*}
R\ul{\mathrm{Hom}}_{\mathsf{CondAb}}(\ul{\bb{Z}_p},\ul{\bb{Z}}) & = \varinjlim_n R\ul{\mathrm{Hom}}_{\mathsf{CondAb}}(\ul{\bb{Z}/p^n\bb{Z}},\ul{\bb{Z}}) \\ & = \varinjlim_n \,(\ul{\bb{Z}/p^n\bb{Z}})[-1] \\ & = (\ul{\bb{Z}[1/p]/\bb{Z}})[-1]
\end{align*}
(see \cite[Rmk.\ 2, p.\ 2]{AMathew} for the first equality), and for its denominator
$$\ul{\bb{Z}[1/q]}=\varinjlim\,(\ul{\bb{Z}}\xrightarrow{\times q}\ul{\bb{Z}}\xrightarrow{\times q}\cdots),$$
we have
\begin{align*}
R\ul{\mathrm{Hom}}_{\mathsf{CondAb}}(\ul{\bb{Z}[1/q]},\ul{\bb{Z}}) & = R\mathrm{lim}\; (R\ul{\mathrm{Hom}}_{\mathsf{CondAb}}(\ul{\bb{Z}},\ul{\bb{Z}})\xleftarrow{\times q}R\ul{\mathrm{Hom}}_{\mathsf{CondAb}}(\ul{\bb{Z}},\ul{\bb{Z}})\leftarrow\cdots) \\ & = R\mathrm{lim}\;(\ul{\bb{Z}}\xleftarrow{\times q}\ul{\bb{Z}}\xleftarrow{\times q}\cdots) \\ & = (\ul{\bb{Z}_q/\bb{Z}})[-1].
\end{align*}
Hence, applying $R\ul{\mathrm{Hom}}_{\mathsf{CondAb}}(\,\cdot\,,\ul{\bb{Z}})$ to the short exact sequence
$$0\to\ul{\bb{Z}[1/q]}\to\ul{\bb{Z}_p}\to\ul{\bb{Z}_p}/\ul{\bb{Z}[1/q]}\to 0$$
we find that
$$R\ul{\mathrm{Hom}}_{\mathsf{CondAb}}(\ul{\bb{Z}_p}/\ul{\bb{Z}[1/q]},\ul{\bb{Z}})=(\ul{\bb{Z}_q}/\ul{\bb{Z}[1/p]})[-2]$$
and that
$\ul{\mathrm{Ext}}^1(B,\ul{\bb{Z}})$, in particular, equals zero.

\section{Conclusion} \label{conclusion}

The project leading to this article began with a desire 
simply to understand the combinatorial phenomena underlying 
considerations of Whitehead's problem in the setting of condensed 
abelian groups, but quickly expanded into the beginnings of a more 
general investigation into connections between condensed mathematics 
and forcing, an investigation that we hope to see further developed 
in the future. These connections have manifested herein in a 
number of ways, notably in
\begin{enumerate}
    \item the recognition that condensed objects can often be 
    viewed as organized presentations of forcing names, and
    \item the observation that interesting phenomena arise via the evaluation of condensed objects at the 
    Stone spaces of complete Boolean algebras significant in set theoretic forcing (cf.\ clause (2) 
    of Theorem~\ref{main_thm}).
\end{enumerate}
In the context of Whitehead's problem, these investigations 
culminated in our Theorem~\ref{forcing_thm} that, if 
$A$ is an abelian group and $\B$ is a complete Boolean algebra 
which forces a certain technical strengthening of 
the failure of $A$ to be Whitehead, then 
$\iExt^1_{\CondAb}(\ul{A}, \ul{\Z})(S(\B)) \neq 0$. We conjecture that this technical strengthening can be foregone:

\begin{conjecture} \label{conjecture}
    Suppose that $A$ is an abelian group, $\B$ is a complete 
    Boolean algebra, and 
    \[
      \Vdash_{\B} ``A \text{ is not Whitehead}".
    \]
    Then $\iExt^1_{\CondAb}(\ul{A}, \ul{\Z})(S(\B)) \neq 0$.
\end{conjecture}
Note that the converse of this conjecture does not hold, as shown in the discussion following Corollary~\ref{forcing_cor}.
We feel that a resolution of this conjecture will shed light 
not just on the relationship between forcing and condensed mathematics, 
but also on considerations of Whitehead's problem in the classical 
setting.

\section*{Acknowledgements}

A portion of this work was carried out while all three authors 
were participating in the Thematic Program on Set Theoretic Methods in 
Algebra, Dynamics and Geometry at the Fields Institute in spring of 2023. 
We thank the Fields Institute for their support and hospitality. 
The results of this paper were presented in detail at the Cornell Logic 
Seminar in Fall 2023. We would like to thank the participants of the 
seminar, particularly Mark Schachner and Justin Moore, for their 
engagement with the material and for a number of helpful remarks, and to thank Peter Scholze for valuable remarks on the paper as well. We thank Radek Honz\'{i}k 
for pointing out a problem with an earlier formulation of Theorem \ref{thm:cohen} 
and Nick Rozenblyum for pointing out a gap in an earlier version of 
Theorem \ref{main_thm}.
Finally, we thank the anonymous referee for a number of helpful suggestions and corrections.

\bibliographystyle{plain}
\bibliography{bib}

\end{document}